\input ./style/arxiv-general.cfg
\documentclass[aop,MSNbibl,dvips]{arximspdf}
\makeatletter
   \@ifpackageloaded{graphicx}{}{\usepackage{graphicx}}
\makeatother
\usepackage{mathbh}

\doi{10.1214/14-AOP978}% Updated by VTEXPTS2LaTeX.exe, 11.12.2014 09:06
\volume{44}
\issue{1}
\pubyear{2016}
\firstpage{444}
\lastpage{478}
\docsubty{FLA}

\makeatletter
\newcommand{\taa}{{\mathop{\rule{0pt}{6pt}\smash{\Delta}}\limits^{\circ}}}
\newcommand{\AAA}{{\mathop{\rule{0pt}{6pt}\smash{A}}\limits^{\circ}}}
\newcommand{\DD}{{\mathop{\rule{0pt}{6pt}\smash{D}}\limits^{\circ}}}
\newcommand{\CC}{{\mathop{\rule{0pt}{6pt}\smash{\mathcal{C}}}\limits^{\circ}}}
\newcommand{\rrvert}{\vert}
\newcommand{\llvert}{\vert}
\newtheorem{theo}{Theorem}
\newproclaim{defi}[theo]{Definition}
\newtheorem{prop}[theo]{Proposition}
\newtheorem{cor}[theo]{Corollary}
\newtheorem{lem}[theo]{Lemma}
\newproclaim{nota}{Notation}
\makeatother

\begin{document}
\begin{frontmatter}

\title{A Curie--Weiss model of self-organized criticality}
\runtitle{A Curie--Weiss model of SOC}

\begin{aug}
% Corresponding author: Matthias Gorny - matthias.gorny@math.u-psud.fr% Updated by VTEXPTS2LaTeX.exe, 15.12.2014 08:41
%Updated by VTEXPTS2LaTeX.exe, 11.12.2014 09:06
\author[A]{\fnms{Rapha\"el}~\snm{Cerf}\ead[label=e1]{rcerf@math.u-psud.fr}}%,
%\author[]{\fnms{}~\snm{}\ead[label=]{}}
\and
\author[A]{\fnms{Matthias}~\snm{Gorny}\corref{}\ead[label=e2]{matthias.gorny@math.u-psud.fr}}
\runauthor{R. Cerf and M. Gorny}
\affiliation{Universit\'e Paris-Sud}
%\dedicated{}
\address[A]{D\'epartement de Math\'ematiques\\
B\^atiment 425\\
Facult\'e des Sciences d'Orsay\\
Universit\'e Paris-Sud\\
F-91405 Orsay Cedex\\
\printead{e1}\\
\phantom{E-mail:\ }\printead*{e2}}
%\address[]{\\\printead{}}
\end{aug}

% HISTORY:
%
\received{\smonth{6} \syear{2013}}% Updated by VTEXPTS2LaTeX.exe,
%11.12.2014 09:06
%
\revised{\smonth{1} \syear{2014}}% Updated by VTEXPTS2LaTeX.exe,
%11.12.2014 09:06

% ABSTRACT
%
\begin{abstract}
We try to design a simple model exhibiting self-organized criticality,
which is amenable to a rigorous mathematical analysis. To this end, we
modify the generalized Ising Curie--Weiss model by implementing an
automatic control of the inverse temperature. For a class of symmetric
distributions whose density satisfies some integrability conditions, we
prove that the sum $S_{n}$ of the random variables behaves as in the
typical critical generalized Ising Curie--Weiss model. The fluctuations
are of order $n^{3/4}$, and the limiting law is $C \exp(-\lambda
x^{4})\,dx$ where $C$ and $\lambda$ are suitable positive constants.
\end{abstract}

% KEYWORDS
\end{frontmatter}

%s1 #&#
\section{Introduction}

In their famous article~\cite{BTW},  Bak, Tang and
Wiesenfeld showed that certain complex systems are naturally attracted
by critical points, without any external intervention. The
amplification of small internal fluctuations can lead to a critical
state and cause a chain reaction leading to a radical change of the
system behavior. These systems exhibit the phenomenon of self-organized
criticality (SOC). Although there is no universal SOC theory, it can be
well understood with the archetype of SOC: the sandpile model, first
introduced in~\cite{BTW}. We consider a pile of sand and the constant
drop of new sand grains, which randomly slide down the slope of sand.
We observe local avalanches with different and unpredictable sizes
which are not proportional to the input. Such phenomenon can be
observed in nature (e.g., forest fires, earthquakes, species evolution).

In general SOC can be observed empirically or simulated on a computer
in various models. However, the mathematical analysis of these models
turns out to be extremely difficult, even for the sandpile model whose
definition is yet simple. Self-organized criticality has been reviewed
in recent works \cite{Aschwanden,Bak,Dhar,Pruessner,Turcotte}. Other
challenging models are the models for forest fires~\cite{RT}, which are
built with the help of percolation process. Some simple models of
evolutions also lead to critical behaviors~\cite{BDF}.

Our goal here is to design a model exhibiting self-organized
criticality, which is as simple as possible, and which is amenable to a
rigorous mathematical analysis. The simplest models exhibiting SOC are
obtained by forcing standard critical transitions into a self-organized
state; see Section~15.4.2 of~\cite{Sornette}. The idea is to start with
a model presenting a phase transition and to create a feedback from the
configuration to the control parameters in order to converge toward a
critical point. The most widely studied model in statistical mechanics,
which exhibits a phase transition and presents critical states, is the
Ising model. Its mean field version is called the Ising Curie--Weiss
model; see Sections~IV.4 and~V.9~of~\cite{Ellis}. It has been extended
to real-valued spins by Ellis and Newman~\cite{EN}, in the so called
generalized Ising Curie--Weiss model. This model is our starting point,
and we will modify it in order to build a system of interacting random
variables, which exhibits a phenomenon of SOC.

Let us first recall the definition and some results on the generalized
Ising Curie--Weiss model. Let $\rho$ be a symmetric probability measure
on $\mathbb{R}$ with positive variance $\sigma^{2}$ and such that
\[
\forall t \geq 0 \qquad\int_{\mathbb{R}}\exp\bigl(tx^{2}
\bigr)\,d\rho (x)<\infty.
\]
The generalized Ising Curie--Weiss model associated to $\rho$ and the
inverse temperature $\beta>0$ is defined through an infinite triangular\vspace*{1pt}
array of real-valued random variables $(X_{n}^{k})_{1\leq k \leq n}$
such that, for all $n \geq1$, $(X^{1}_{n},\ldots,X^{n}_{n})$ has the
distribution
\[
d \mu_{n,\rho,\beta}(x_{1},\ldots,x_{n})=\frac{1}{Z_{n}(\beta)}
\exp \biggl(\frac
{\beta}{2} \frac{(x_{1}+\cdots+x_{n})^{2}}{n} \biggr) \prod
_{i=1}^{n} \,d\rho (x_{i}),
\]
where $Z_{n}(\beta)$ is a normalization. For any $n \geq1$, we set
$S_{n}=X^{1}_{n}+\cdots+X^{n}_{n}$. When $\rho=(\delta_{-1}+\delta
_{1})/2$, we
recover the classical Ising Curie--Weiss model.

We denote by $L$ the log-Laplace of $\rho$ (see Appendix~\ref{appA}). Ellis and
Eisele have shown in~\cite{EE} that, if $L^{(3)}(t) \leq0$ for any
$t\geq0$, then there exists a map $m$ which is null on $]0,1/\sigma^{2}]$,
real analytic and positive on $]1/\sigma^{2},+\infty[$ and such that
\[
\frac{S_{n}}{n} \mathop{\longrightarrow}^{\mathcal{L}}_{n \to\infty
}\cases{
\delta_{0}, & \quad\mbox{if }$\beta\leq1/\sigma^{2}$,\vspace *{3pt}
\cr
\displaystyle\tfrac{1}{2}(\delta_{-m(\beta)}+\delta_{m(\beta)}), & \quad\mbox{if }$\beta>1/
\sigma^{2}$.}
\]
The point $1/\sigma^{2}$ is a critical value, and the function $m$ cannot
be extended analytically around $1/\sigma^{2}$. The main result
of~\cite{EN} states that, if $\beta<1/\sigma^{2}$, then, under~$\mu_{n,\rho,\beta}$,
\[
\frac{S_{n}}{\sqrt{n}} \mathop{\longrightarrow}^{\mathcal{L}}_{n
\to\infty}
\mathcal{N} \biggl(0,\frac{\sigma^{2}}{1-\beta\sigma^{2}} \biggr).
\]
If $\beta=1/\sigma^{2}$, then there exists $k \in\mathbb
{N}\setminus\{0,1\}$ and $\lambda
>0$ such that, under~$\mu_{n,\rho,\beta}$,
\[
\frac{S_{n}}{n^{1-1/2k}} \mathop{\longrightarrow}^{\mathcal{L}}_{n \to
\infty}
C_{k,\lambda} \exp \biggl(-\lambda\frac
{s^{2k}}{(2k)!} \biggr)\,ds,
\]
where $C_{k,\lambda}$ is a normalization. This is a consequence of
Theorem~2.1 of~\cite{EN} and some properties of $m$ explained in~\cite{EE} implying that the function
$s \longmapsto L(s\sqrt{\beta})-s^{2}/2$
has a unique maximum at~$0$ whenever $\beta\leq1/\sigma^{2}$; see
Section~V.2 of~\cite{M2GOR} for the details.

We will transform the previous probability distribution in order to
obtain a model which presents a phenomenon of self-organized
criticality, that is, a model which evolves toward the critical state
$\beta=1/\sigma^{2}$ of the previous model. More precisely, the critical
generalized Ising Curie--Weiss model is the model where
$(X^{1}_{n},\ldots,X^{n}_{n})$ has the distribution
\[
\frac{1}{Z_{n}} \exp \biggl(\frac{(x_{1}+\cdots+x_{n})^{2}}{2n\sigma
^{2}} \biggr) \prod
_{i=1}^{n} \,d\rho(x_{i}).
\]
We wish to build a model which converges to a critical state for every
distribution~$\rho$ and which does not rely on any specific a priori
information on $\rho$. We search an automatic control of the inverse
temperature $\beta$, which would be a function of the random variables in
the model, so that when $n$ goes to $+ \infty$, $\beta$ converges toward
the critical value of the model. We start with the following
observation: if $(Y_{n})_{n \geq1}$ is a sequence of independent
random variables with identical distribution $\rho$, then, by the law of
large numbers,
\[
\frac{Y_{1}^{2}+\cdots+Y_{n}^{2}}{n} \mathop{\longrightarrow}_{n \to
\infty}
\sigma^{2}\qquad\mbox{a.s.}
\]
This convergence provides us with an estimator of $1/\sigma^2$. If we
believe that a similar convergence holds in the generalized Ising
Curie--Weiss model, then we are tempted to
\textit{replace  $\beta$ by $n (x_{1}^{2}+\cdots+x_{n}^{2})^{-1}$} in the
distribution
\[
\frac{1}{Z_{n}} \exp \biggl(\frac{\beta}{2}\frac{(x_{1}+\cdots
+x_{n})^{2}}{n} \biggr)
\prod_{i=1}^{n} \,d\rho(x_{i}).
\]
Hence the model we consider in this paper is given by the distribution
\[
\frac{1}{Z_{n}} \exp \biggl(\frac{1}{2}\frac{(x_{1}+\cdots
+x_{n})^{2}}{x_{1}^{2}+\cdots+x_{n}^{2}} \biggr)\prod
_{i=1}^{n} \,d\rho (x_{i}).
\]
The previous considerations suggest that this model should evolve
spontaneously toward a critical state. We will prove rigorously that
our model indeed exhibits a phenomenon of self-organized criticality.
However, our model is a toy model which is certainly much less complex
than other famous fundamental models of SOC like the sandpile model.

Our main result (Theorem~\ref{theoFluctuations}) states that if $\rho$
has an even density satisfying some integrability condition, then,
asymptotically, the sum $S_{n}$ of the random variables behaves as in
the typical critical generalized Ising Curie--Weiss model: if $\mu_{4}$
denotes the fourth moment of $\rho$, then
\[
\frac{\mu_{4}^{1/4}S_{n}}{\sigma^{2}n^{3/4}} \mathop{\longrightarrow }^{\mathcal{L}}_{n \to\infty} \biggl(
\frac{4}{3} \biggr)^{1/4}\Gamma \biggl(\frac
{1}{4}
\biggr)^{-1} \exp \biggl(-\frac{s^{4}}{12} \biggr)\,ds.
\]

This fluctuation result shows that our model is a self-organized model
exhibiting critical behavior. Indeed it has the same behavior as the
critical generalized Ising Curie--Weiss model, and by construction, it
does not depend on any external parameter. In this sense, we can
conclude that this is a Curie--Weiss model of self-organized criticality.

Our result presents an unexpected universal feature. For any
distribution~$\rho$, which has an even density satisfying some
integrability hypothesis, the fluctuations of $S_{n}$ are of order
$n^{3/4}$. This is in contrast to the situation in the critical
generalized Ising Curie--Weiss model: at the critical point, the
fluctuations are of order $n^{1-1/2k}$, where $k$ depends on the
distribution $\rho$.
We stress also that our integrability conditions on $\rho$ are weaker
than those
of~\cite{EN}. For instance, our result holds for any centered Gaussian
measure on $\mathbb{R}$. The Gaussian case of our model can be handled
with the
help of an explicit computation~\cite{GORGaussCase}.

The main new technical ingredient of the proof is the following
inequality. Let $Z$ be a random variable with distribution $\rho$, and
let $I$ denote the
Cram\'er transform of $(Z,Z^{2})$, given by
\[
\forall(x,y) \in\mathbb{R}^{2} \qquad I(x,y)=\sup
_{(u,v) \in
\mathbb{R}^{2}} \biggl\{ xu+yv-\ln\int_{\mathbb{R}}e^{uz+vz^{2}}
\,d\rho(z) \biggr\}.
\]
If $\rho$ is symmetric and there exists $v>0$ such that $E(\exp
(vZ^2))<+\infty$, then
\[
\forall(x,y) \in\mathbb{R}^{2} \qquad I(x,y) \geq \frac{x^2}{2y},
\]
and the equality holds only at $(0,\sigma^2)$. We explain in the
heuristics at the end of Section~\ref{theoCV} why this inequality is
crucial to the proof of our main results.

In Section~\ref{model} we properly define our model. We state our main
results and the strategy for proving them in Section~\ref{theoCV}. Next
we split the proofs in the remaining Sections \ref{MinimaI-F}--\ref{ProofFluctuations}.
In the \hyperref[appA]{Appendix}, we recall some generalities on
the Cram\'er transform and large deviations.

%s2 #&#
\section{The model}
\label{model}

Let $\rho$ be a probability measure on $\mathbb{R}$, which is not the
Dirac mass
at 0. We consider an infinite triangular array of real-valued random
variables $(X_{n}^{k})_{1\leq k \leq n}$ such that for all $n \geq1$,
$(X^{1}_{n},\ldots,X^{n}_{n})$ has the distribution $\tilde{\mu}_{n,\rho}$, where
\[
d\tilde{\mu}_{n,\rho}(x_{1},\ldots,x_{n})=
\frac{1}{Z_{n}}\exp \biggl(\frac{1}{2}\frac{(x_{1}+\cdots+x_{n})^{2}}{x_{1}^{2}+\cdots
+x_{n}^{2}} \biggr)
\mathbh{1}_{\{x_{1}^{2}+\cdots+x_{n}^{2}>0\}} \prod_{i=1}^{n}\,d
\rho(x_{i}),
\]
with
\[
Z_{n}=\int_{\mathbb{R}^{n}}\exp \biggl(\frac{1}{2}
\frac
{(x_{1}+\cdots
+x_{n})^{2}}{x_{1}^{2}+\cdots+x_{n}^{2}} \biggr)\mathbh{1}_{\{
x_{1}^{2}+\cdots +x_{n}^{2}>0\}} \prod
_{i=1}^{n}\,d\rho(x_{i}).
\]
We define $S_{n}=X^{1}_{n}+\cdots+X^{n}_{n}$ and
$T_{n}=(X^{1}_{n})^{2}+\cdots+(X^{n}_{n})^{2}$.

The indicator function in the density of the distribution $\tilde
{\mu}_{n,\rho}$ helps to avoid any problem of definition if $\rho(\{
0\})$
is positive, since, if $\rho(\{0\})>0$, the event $\{x_{1}^{2}+\cdots
+x_{n}^{2}=0\}$ may occur with positive probability. We notice that,
unlike the generalized Ising Curie--Weiss model, our model is defined
for any probability measure. Indeed $x \longmapsto x^{2}$ is a convex
function, and therefore
\[
\forall(x_{1},\ldots,x_{n}) \in\mathbb{R}^{n}
\qquad \Biggl(\sum_{i=1}^{n}x_{i}
\Biggr)^{2}=n^{2} \Biggl(\sum_{i=1}^{n}
\frac
{x_{i}}{n} \Biggr)^{2}\leq n \sum_{i=1}^{n}
x_{i}^{2}.
\]
Thus for any $n \geq1$, $1 \leq Z_{n} \leq e^{n/2}< +\infty$.

If we choose $\rho=(\delta_{-1}+\delta_{1})/2$, we obtain the
classical Ising
Curie--Weiss model at the critical value.

%s3 #&#
\section{Convergence theorems}
\label{theoCV}

We state here our main results. By the classical law of large numbers,
if $\rho$ is centered and has variance $\sigma^{2}$, then under $\rho
^{\otimes
n}$, $(S_{n}/n,T_{n}/n)$ converges in probability toward $(0,\sigma^{2})$.
The next theorem shows that under\vspace*{1pt} the law~$\tilde{\mu}_{n,\rho}$,
given certain conditions, $(S_{n}/n,T_{n}/n)$ also converges in
probability to~$(0,\sigma^{2})$.

%th1 #&#
\begin{theo}\label{theoCVproba}
Let $\rho$ be a symmetric probability measure on $\mathbb{R}$ with
positive variance $\sigma^{2}$ and such that
\[
\exists v_0>0 \qquad\int_{\mathbb{R}}e^{v_{0}z^{2}}
\,d\rho (z)<+\infty.
\]
We suppose that one of the following conditions holds:
\begin{longlist}[(a)]
\item[(a)] $\rho$ has a density.

\item[(b)] $\rho$ is the sum of a finite number of Dirac masses.

\item[(c)] There exists $c>0$ such that $\rho(]0,c[)=0$.

\item[(d)] $\rho(\{0\})<1/\sqrt{e}$.
\end{longlist}
Then, under $\tilde{\mu}_{n,\rho}$, $(S_{n}/n,T_{n}/n)$
converges in probability toward $(0,\sigma^{2})$.
\end{theo}

By the classical central limit theorem, under $\rho^{\otimes n}$,
$S_{n}/\sqrt{n}$ converges in distribution to a normal distribution
with mean zero and variance $\sigma^{2}$. The following theorem shows that
given certain conditions, under $\tilde{\mu}_{n,\rho}$,
$S_{n}/n^{3/4}$ converges toward a specific distribution.

%th2 #&#
\begin{theo}\label{theoFluctuations}
Let $\rho$ be a probability measure on $\mathbb{R}$
having a density
$f$ which satisfies:
\begin{longlist}[(a)]
\item[(a)] $f$ is even.
\item[(b)] There exists $v_{0}>0$ such that
\[
\int_{\mathbb{R}}e^{v_{0}z^{2}}f(z)\,dz<+\infty.
\]
\item[(c)] There exists $p \in\,]1,2]$ such that
\[
\int_{\mathbb{R}^{2}}f^{p}(x+y)f^{p}(y)|x|^{1-p}\,dx\,dy<+\infty.
\]
Let $\sigma^{2}$ be the variance of $\rho$, and let $\mu
_{4}$ be
the fourth moment of $\rho$. We have
\[
\frac{\mu_{4}^{1/4}S_{n}}{\sigma^{2}n^{3/4}} \mathop {\longrightarrow}^{\mathcal{L}}_{n \to\infty}
\biggl(\frac{4}{3} \biggr)^{1/4}\Gamma \biggl(\frac{1}{4}
\biggr)^{-1} \exp \biggl(-\frac{s^{4}}{12} \biggr)\,ds.
\]
\end{longlist}
\end{theo}

The convergence can equivalently be rewritten as
\[
\frac{S_{n}}{n^{3/4}} \mathop{\longrightarrow}^{\mathcal{L}}_{n \to\infty
} \biggl(\frac{4\mu_{4}}{3\sigma^{8}} \biggr)^{1/4}\Gamma \biggl(\frac{1}{4}
\biggr)^{-1} \exp \biggl(-\frac{\mu_{4}}{12 \sigma
^{8}}s^{4} \biggr)\,ds.
\]
We prove this convergence in Section~\ref{ProofFluctuations}. The
following corollary is a version of Theorem~\ref{theoFluctuations} with
a hypothesis which is weaker but easier to check.

%co3 #&#
\begin{cor}\label{corFluctuations}
Let $\rho$ be a probability measure on $\mathbb{R}$ with
an even and
bounded density $f$ such that
\[
\exists v_{0}>0 \qquad\int_{\mathbb{R}}e^{v_{0}z^{2}}
\,d\rho (z)<+\infty.
\]
Let $\sigma^{2}$ be the variance of $\rho$, and let $\mu_{4}$ be the fourth
moment of $\rho$. Then
\[
\frac{\mu_{4}^{1/4}S_{n}}{\sigma^{2}n^{3/4}} \mathop{\longrightarrow}^{\mathcal{L}}_{n \to\infty}
\biggl(\frac
{4}{3} \biggr)^{1/4}\Gamma \biggl(\frac
{1}{4}
\biggr)^{-1} \exp \biggl(-\frac{s^{4}}{12} \biggr)\,ds.
\]
\end{cor}

\begin{pf}
We check that the hypotheses of the corollary imply the
condition~(c) of Theorem~\ref{theoFluctuations}. We have
\begin{eqnarray*}
&&\int_{\mathbb{R}^{2}} f^{3/2}(x+y)f^{3/2}(y)|x|^{-1/2}
\,dx\,dy
\\
&&\qquad=\int_{[-1,1]\times\mathbb{R}} \frac{f^{3/2}(x+y)f^{3/2}(y)}{|x|^{1/2}}\,dx\, dy\\
&&\qquad\quad{}+\int
_{[-1,1]^c\times\mathbb{R}} \frac
{f^{3/2}(x+y)f^{3/2}(y)}{|x|^{1/2}}\, dx\,dy
\\
&&\qquad\leq\|f\|_{\infty}^{3/2}\int_{[-1,1]\times\mathbb{R}}
\frac
{f^{3/2}(y)}{|x|^{1/2}}\,dx\,dy+\int_{[-1,1]^c\times\mathbb{R}} f^{3/2}(x+y)f^{3/2}(y)
\,dx\,dy
\\
&&\qquad \leq\|f\|_{\infty}^{3/2} \biggl(\int_{\mathbb{R}}\bigl|f(x)\bigr|^{3/2}
\, dx \biggr) \biggl(\int_{-1}^1
\frac{dx}{|x|^{1/2}} \biggr)+ \biggl(\int_{\mathbb
{R}}\bigl|f(x)\bigr|^{3/2}\, dx \biggr)^{2}.
\end{eqnarray*}
The second inequality is obtained by applying Fubini's theorem. These
terms are finite since
\[
\int_{\mathbb{R}}\bigl|f(x)\bigr|^{3/2}\,dx \leq\|f
\|^{1/2}_{\infty}\int_{\mathbb{R}}f(x)\,dx=\| f
\|^{1/2}_{\infty}<+\infty.
\]
Thus, with $p=3/2 \in\,]1,2]$, the function $(x,y)\longmapsto
f^{p}(x+y)f^{p}(y)|x|^{1-p}$ is integrable.
\end{pf}

For instance, if $\rho$ has a bounded support and a density which is even
and continuous on it, then the hypotheses of the theorem are fulfilled.

%\label{heuristiques}

We end this section by computing the law of $(S_{n}/n,T_{n}/n)$ under
$\tilde{\mu}_{n,\rho}$ and explaining the strategy for\vspace*{1pt} proving these
results. We denote by $\tilde{\nu}_{n,\rho}$ the
law of $(S_{n}/n,T_{n}/n)$ under $\rho^{\otimes n}$. We have
\[
\forall(x_{1},\ldots,x_{n}) \in\mathbb{R}^{n}
\qquad\frac{(x_{1}+\cdots
+x_{n})^{2}}{x_{1}^{2}+\cdots+x_{n}^{2}}=n\frac{((x_{1}+\cdots
+x_{n})/n)^{2}}{(x_{1}^{2}+\cdots+x_{n}^{2})/n}.
\]
Hence, for any bounded measurable function $f \dvtx  \mathbb
{R}^{2}\longrightarrow\mathbb{R}$,
\[
\mathbb{E}_{\tilde{\mu}_{n,\rho}} \biggl(f \biggl(\frac
{S_{n}}{n},\frac
{T_{n}}{n}
\biggr) \biggr)=\frac{1}{Z_{n}}\int_{\mathbb
{R}^{2}}f(x,y)\exp
\biggl(\frac{nx^{2}}{2y} \biggr)\mathbh{1}_{\{y>0\}}\,d\tilde{\nu
}_{n,\rho}(x,y).
\]
By convexity of $t\longmapsto t^{2}$, we have $S_{n}^{2}\leq nT_{n}$
for any $n \geq1$. We define
\[
\Delta=\bigl\{ (x,y) \in\mathbb{R}^{2} \dvtx  x^{2} \leq y
\bigr\} \quad\mbox{and} \quad\Delta^{*}=\Delta\setminus\bigl\{(0,0)\bigr\}.
\]
Thus $\tilde{\nu}_{n,\rho} (\Delta^{c} )=0$.
Therefore we have
the following proposition:

%pr4 #&#
\begin{prop}\label{loiSnTn1}
Under $\tilde{\mu}_{n,\rho}$, the law of $(S_{n}/n,T_{n}
/n)$ is
\[
\frac{ {\exp(({nx^{2}})/({2y}))\mathbh
{1}_{\Delta^{  *}}(x,y)\,d\tilde{\nu}_{n,\rho
}(x,y)}}{{\int_{\Delta^{
*}}\exp (({ns^{2}})/({2t}))\,d\tilde{\nu}_{n,\rho
}(s,t)}}.
\]
\end{prop}

We denote by $\nu_{\rho}$ the law of $(Z,Z^{2})$ where $Z$ is a random
variable with distribution $\rho$. The log-Laplace $\Lambda$ of $\nu
_{\rho}$ is
the map defined on $\mathbb{R}^2$ by
\[
\forall(u,v) \in\mathbb{R}^{2} \qquad\Lambda(u,v)=\ln\int
_{\mathbb{R}^{2}}e^{us+vt}\,d\nu _{\rho}(s,t)=\ln\int
_{\mathbb{R}}e^{uz+vz^{2}}\,d\rho(z),
\]
and the Cram\'er transform $I$ of $\nu_{\rho}$ is defined on $\mathbb
{R}^2$ by
\[
\forall(x,y) \in\mathbb{R}^2 \qquad I(x,y)=\sup_{(u,v)\in\mathbb{R}^{2}}
\bigl(xu+yv-\Lambda(u,v) \bigr).
\]
For $n \geq1$, under $\rho^{\otimes n}$, $(S_{n}/n,T_{n}/n)$ is the sum
of $n$ independent and identically distributed random variables with
distribution $\nu_{\rho}$. We refer to Appendix~\ref{appB} for some definitions
and results on large deviations, especially Cram\'er's theorem
(Theorem~\ref{Cramer}) which states that if $\Lambda$ is finite in the
neighborhood of~$(0,0)$, then $I$ is a good rate function, and
$(\tilde{\nu}_{n,\rho})_{n \geq1}$ satisfies the large deviations
principle with speed~$n$, governed by $I$.

Here is a classical heuristic on large deviations, suggested by a
consequence of Varadhan's lemma (see Theorem~II.7.2~of~\cite{Ellis}):
as $n$ goes to $+\infty$, the law of $(S_{n}/n,T_{n}/n)$ under
$\tilde{\mu}_{n,\rho}$ concentrates exponentially fast on the minima
on $\Delta^{*}$ of the function
\[
G=I-F- \inf_{\Delta^{ *}} (I-F),
\]
where $F$ is the map defined by
\[
\forall(x,y)\in\mathbb{R}\times\mathbb{R}\setminus\{0\} \qquad F(x,y)=
\frac
{x^{2}}{2y}.
\]
If $G$ has a unique minimum at $(x_{0},y_{0}) \in\Delta^{ *}$, then under
$\tilde{\mu}_{n,\rho}$, $(S_{n}/n,T_{n}/n)$ converges in probability
to $(x_{0},y_{0})$. Moreover, the large deviations principle suggests
that for $n$ large enough, $\tilde{\nu}_{n,\rho}$ can
roughly be approximated by the distribution $C_{n}\exp(-nI(x,y))\,dx\,
dy$ where $C_{n}$ is a normalizing constant. Thus, for each bounded
continuous function $h$ and $\alpha,\beta>0$,
\begin{eqnarray*}
\mathbb{E}_{\tilde{\mu}_{n}} \biggl(h \biggl(\frac
{S_{n}-nx_{0}}{n^{1-\alpha}} \biggr) \biggr)
&\approx & \frac{\int_{\Delta^{ *}}h((x-x_{0})
n^{\alpha
})\exp (-nG(x,y) )\,dx\,dy}{\int_{\Delta
^{ *}}\exp
 (-nG(x,y) )\,dx\,dy}
\\
& \approx & \frac{\int_{\Delta^{ *}}h(x)\exp
(-nG
(xn^{-\alpha}+x_{0},yn^{-\beta}+y_{0} ) )\,dx\,
dy}{\int_{\Delta^{ *}}\exp (-nG (xn^{-\alpha
}+x_{0},yn^{-\beta}+y_{0}
) )\,dx\,dy}.
\end{eqnarray*}
We use then Laplace's method. The key point is the study of the
function $G$ in the neighborhood of its minimum $(x_{0},y_{0})$. We
find four positive values $A$, $B$, $a\in\mathbb{N}$ and $b\in
\mathbb{N}$ such that,
uniformly on a neighborhood of $(x_{0},y_{0})$,
\[
-nG \bigl(xn^{-1/a}+x_{0},yn^{-1/b}+y_{0}
\bigr)\mathop {\longrightarrow}_{n \to\infty}-Ax^{a}-By^{b}.
\]

We prove that $I-F$ has a unique minimum at $(0,\sigma^2)$ on $\Delta
^{ *}$ in
Section~\ref{MinimaI-F}. Next we give the proof of Theorem~\ref{theoCVproba} in Section~\ref{ProofCVproba}, with the help of a variant
of Varadhan's lemma. Finally we compute the expansion of $I-F$ around
$(0,\sigma^2)$ in Section~\ref{Expansion}, and we prove Theorem~\ref
{theoFluctuations} with Laplace's method in Section~\ref{ProofFluctuations}. Throughout these proofs we use some general
results on the Cram\'er transform, stated in Appendix~\ref{appA}.

%s4 #&#
\section{Minimum of $I-F$ on \texorpdfstring{$\Delta^{*}$}{$Delta^{*}$}}
\label{MinimaI-F}

Let $\rho$ be a symmetric probability measure on $\mathbb{R}$. In this
section, we will use Proposition~\ref{Dadmissible} in the \hyperref[appA]{Appendix} to
show an inequality between $I$ and $F$.

We denote by $\nu_{\rho}$ the distribution of $(Z,Z^{2})$ when $Z$ is a
random variable with law $\rho$. If the support of $\rho$ contains at least
three points, then $\nu_{\rho}$ is a nondegenerate measure on
$\mathbb{R}^{2}$;
see the first paragraphs of Appendix~\ref{appA}. We denote by $\mathcal{C}$ the convex
hull of the set
$\{ (x,x^{2})\dvtx x \mbox{ is in the support of } \rho \}$.
The function
\[
\Lambda\dvtx  (u,v)\in\mathbb{R}^{2} \longmapsto\ln\int
_{\mathbb{R}}e^{uz+vz^{2}}\,d\rho(z)
\]
is the log-Laplace of $\nu_{\rho}$, and its domain of definition
$D_{\Lambda}$
contains $\mathbb{R}\times\,]{-}\infty,0[$; thus its interior is
nonempty. Let
$I$ be the Cram\'er transform of $\nu_{\rho}$. We denote by $D_I$ its
domain of definition and by $A_{I}=
\nabla\Lambda({\mathop{\rule{0pt}{6pt}\smash{D}}\limits^{\circ}}_{\Lambda})$
its admissible domain; see Definition~\ref{adm}
in the \hyperref[appA]{Appendix}.

Using Jensen's inequality, we get that $I(0,\sigma^{2})=0$. Moreover the
infimum of $I-F$ on $\Delta^{ *}$ belongs\vspace*{1pt} to $[-1/2,0]$. The function $I$
is even in the first variable. Indeed, if $(x,y) \in\mathbb{R}^{2}$, then
\begin{eqnarray*}
I(-x,y)&=&\sup_{(u,v) \in\mathbb{R}^{2}} \biggl(-xu+yv-\ln\int
_{\mathbb{R}
}e^{uz+vz^{2}}\,d\rho(z) \biggr)
\\
&=&\sup_{(u,v) \in\mathbb{R}^{2}} \biggl(xu+yv-\ln\int_{\mathbb
{R}}e^{-uz+vz^{2}}
\,d\rho (z) \biggr)=I(x,y).
\end{eqnarray*}
Assume that $I-F$ has a unique minimum $(x_{0},y_{0})$ on $\Delta^{ *}$.
Then $(-x_{0},y_{0})$ is also a minimum of $I-F$. The uniqueness of the
minimum implies that $x_{0}=0$ so that $I-F$ is nonnegative on $\Delta^{*}$.
Finally, since $I(0,\sigma^{2})=0$, we have $y_{0}=\sigma^{2}$.

Consider first the case of a Bernoulli distribution for which $\nu
_{\rho
}$ is degenerate. Let $c>0$. Suppose that $\rho=(\delta_{-c}+\delta
_{c})/2$. The
law $\rho$ is centered, and its variance is $c^{2}$. We can compute
$\Lambda$
and $I$ explicitly in the following way:
\[
\forall(u,v) \in\mathbb{R}^{2} \qquad\Lambda(u,v)=vc^{2}+
\ln \operatorname{cosh}(uc).
\]
For any $(x,y)\notin[-c,c]\times\{c^{2}\}$, $I(x,y)=+\infty$ and
\[
\forall x\in\,]{-}c,c[ \qquad I\bigl(x,c^{2}\bigr)=\frac{1}{2c}
\bigl((c+x)\ln (c+x)+(c-x)\ln(c-x) \bigr)-\ln c.
\]

The study of the function $x \longmapsto I(x,c^{2})-x^{2}/(2c^{2})$
shows that, in the Bernoulli case, $I-F$ has a unique minimum at
$(0,\sigma^{2})$. More generally we have the following lemma:

%le5 #&#
\begin{lem}\label{casbinomial}
Let $c>0$. We define
\[
\phi_{c} \dvtx x \in\mathbb{R}\longmapsto\sup_{u \in\mathbb{R}}
\bigl( ux-\ln\operatorname{cosh}(uc) \bigr).
\]
The function $x\longmapsto\phi_{c}(x)-x^{2}/(2c^{2})$
is increasing on $[0,c]$, decreasing on $[-c,0]$ and null at~$0$.
\end{lem}

Notice that the Bernoulli case is special since if $X$ is a random
variable with distribution $\rho=(\delta_{-c}+\delta_{c})/2$, then
$X^{2}=c^{2}$
almost surely. Thus
\begin{eqnarray*}
&& \frac{1}{Z_{n}}\exp \biggl(\frac{1}{2}\frac{(x_{1}+\cdots
+x_{n})^{2}}{x_{1}^{2}+\cdots+x_{n}^{2}} \biggr)
\mathbh{1}_{\{x_{1}^{2}+\cdots +x_{n}^{2}>0\}} \prod_{i=1}^{n}\,d
\rho(x_{i})
\\
&& \qquad =\frac{1}{Z_{n}(1/c^{2})}\exp \biggl(\frac{(x_{1}+\cdots
+x_{n})^{2}}{2nc^{2}} \biggr) \prod
_{i=1}^{n}\,d\rho(x_{i}).
\end{eqnarray*}
This is exactly the classical Curie--Weiss model at the critical point.

In the following, we suppose that the support of $\nu_{\rho}$
contains at
least three distinct points. We first show that if $D_{\Lambda}$ is an open
subset of $\mathbb{R}^{2}$, then $I-F$ has a unique minimum at
$(0,\sigma^{2})$. To
this end,\vspace*{-1pt} we use Proposition~\ref{Dadmissible} in the \hyperref[appA]{Appendix} which
states that $I$ is differentiable on $A_{I}=\DD_{I}=\CC$. Moreover,
if $(x,y)\longmapsto(u(x,y),v(x,y))$ is the inverse function of
$\nabla\Lambda$, then
\[
\forall(x,y) \in \DD_{I} \qquad
\frac{\partial I}{\partial
x}(x,y)=u(x,y).
\]
If we show that $u(x,y) >x/y$ for any $x,y>0$, then by integrating this
inequality,
\[
\forall(x,y) \in\DD_{I} \qquad 0\leq
\varepsilon<x \quad \Longrightarrow \quad I(x,y)-\frac{x^{2}}{2y}> I(\varepsilon,y)-
\frac{\varepsilon
^{2}}{2y}.
\]
To obtain that $I-F$ has a unique minimum at $(0,\sigma^{2})$, it is enough
to extend this inequality to the boundary points of $D_{I}$ (if they
exist). We conclude by using the fact that $I$ is even in its first variable.

The following lemma is the key result to establish the uniqueness of
the minimum of $I-F$, when $\rho$ is symmetric.

%le6 #&#
\begin{lem}
\label{lemxy}
Let $\rho$ be a symmetric probability measure whose support
contains at least three points. For $(x,y) \in A_I$, we have $u(x,y)=0$
if $x=0$ and
\begin{eqnarray*}
u(x,y) &>&  \frac{x}{y} \qquad\mbox{if } x>0,
\\
u(x,y) &<&  \frac{x}{y} \qquad\mbox{if }x<0.
\end{eqnarray*}
\end{lem}

\begin{pf}
The vector $(u,v)=(u(x,y),v(x,y))$ verifies
\[
(x,y)=\nabla\Lambda(u,v)= \biggl(\frac{ {\int_{\mathbb
{R}}ze^{uz+vz^{2}}
\,d\rho(z)}}{ {\int_{\mathbb{R}}e^{uz+vz^{2}}\,d\rho
(z)}},\frac
{ {\int_{\mathbb{R}}z^{2}e^{uz+vz^{2}}\,d\rho
(z)}}{ {\int_{\mathbb{R}}e^{uz+vz^{2}}\,d\rho(z)}} \biggr).
\]
The distribution $\rho$ is symmetric, thus
\[
\int_{\mathbb{R}}ze^{uz+vz^{2}}\,d\rho(z)=\int
_{0}^{+\infty}2z \operatorname{sinh}(uz)e^{vz^{2}}\,d
\rho(z).
\]
This formula shows that $u$ and $x$ have the same sign. Moreover for
any $z \geq0$, $\operatorname{tanh}(z)\leq z$. Thus if $x>0$, then $\operatorname{sinh}(uz)\leq uz\operatorname{cosh}(uz)$. The equality holds if and only if
$uz=0$. Therefore, using the symmetry of $\rho$,
\[
x < u \frac{\int_{0}^{+\infty}2z^{2} \operatorname{cosh}(uz)e^{vz^{2}}\,d\rho(z)}{\int_{\mathbb
{R}}e^{uz+vz^{2}}\,d\rho
(z)}=u \frac{\int_{\mathbb{R}}z^{2}
e^{uz+vz^{2}}\,d\rho
(z)}{\int_{\mathbb{R}}e^{uz+vz^{2}}\,d\rho(z)}=uy.
\]
Since $x>0$, $u>0$ and $y>0$, we conclude that $u > x/y$.
Similarly, we show that if $x<0$, then $u < x/y$.
\end{pf}

We can now prove the following inequality:

%pr7 #&#
\begin{prop}\label{inegaliteI-F.0}
If $\rho$ is a symmetric probability measure on $\mathbb
{R}$ with
positive variance $\sigma^{2}$ and such that $D_{\Lambda}$ is an open
subset of
$\mathbb{R}^{2}$, then
\[
\forall(x,\varepsilon,y) \in\mathbb{R}\times\mathbb{R}\times \mathbb{R}
\setminus\{0\} \qquad 0\leq \varepsilon<x  \quad \Longrightarrow \quad I(x,y)-\frac{x^{2}}{2y}
\geq I(\varepsilon,y)-\frac
{\varepsilon^{2}}{2y}.
\]
This inequality is strict if $(\varepsilon,y) \in\DD_{I}$.
\end{prop}

The inequality is also true for $x<\varepsilon\leq0$ since $I$ is
even in
its first variable. In Corollary~\ref{inegaliteI-F}, we shall extend
the inequality to any symmetric distribution on $\mathbb{R}$.

\begin{pf*}{Proof of Proposition~\protect\ref{inegaliteI-F.0}}
We have already treated the Bernoulli case. We assume
next that the support of $\rho$ contains at least three points. The
Cram\'er transform~$I$ is $\mathrm{C}^{\infty}$ on $\DD_{I}$ and
\[
\forall(x,y) \in\DD_{I} \qquad
\frac{\partial I}{\partial
x}(x,y)=u(x,y).
\]
Let us examine the structure of the set $D_{I}$. We put
\[
\forall y>0 \qquad D_{I,y}=\bigl\{ x \in\mathbb{R}\dvtx (x,y) \in
D_I \bigr\}
\]
(see Figure~\ref{fig1}).
Let $y>0$ be such that $(x,y) \in\DD_{I}$ for some $x \in\mathbb{R}$. The
set $D_{I,y}$ is a convex subset of $\mathbb{R}$. Moreover $x
\longmapsto
I(x,y)$ is even, therefore $\DD_{I,y}$ (the interior of $D_{I,y}$ as a
subset of $\mathbb{R}$) is an open interval $]{-}a(y),a(y)[$ with $a(y)
\in
[0,\sqrt{y}]$. Lemma~\ref{lemxy} implies that $u(t,y)> t/y$ for any
$t\in \,]0,a(y)[$. Thus, for any $x \in\DD_{I,y} \,\cap \,]0,+\infty[$,
\[
\forall\varepsilon\in[0,x[ \qquad I(x,y)-I(\varepsilon,y)=\int
_{\varepsilon}^{x}u(t,y)\, dt>\int_{\varepsilon}^{x}
\frac{t}{y}\,dt=\frac{x^{2}}{2y}-\frac
{\varepsilon^{2}}{2y}.
\]
There is no problem of definition at $y=0$ since $\DD_{I} \subset\Delta
^{ *}$ does not contain $\mathbb{R}\times\{0\}$ and $\DD_{I,0}=\varnothing$. Moreover
\[
x \longmapsto\frac{I(x,y)-I(\varepsilon,y)}{x-\varepsilon}
\]
is nondecreasing on $D_{I,y}\setminus\{\varepsilon\}$ since $I$ is convex.
Therefore, if $-a(y)$ and~$a(y)$ belong to $D_{I,y}$, then the previous
inequality extends to $x=-a(y)$ and $x=a(y)$.

%f1
%
%f1 #&#
\begin{figure}

\includegraphics{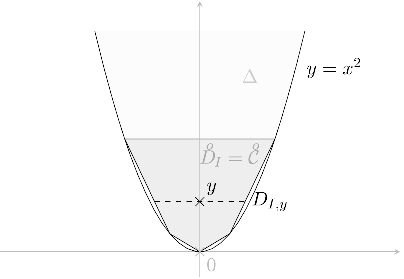}

\caption{Case where $\rho$ is symmetric discrete and charges 5 points.}\label{fig1}
\end{figure}

We have shown that
\[
\forall(x,y) \in D_{I} \qquad y>0, 0\leq\varepsilon<x
\quad\Longrightarrow \quad I(x,y)-I(\varepsilon,y) > \frac{x^{2}}{2y}-\frac{\varepsilon
^{2}}{2y},
\]
except for the points $(x,y)$ of the superior and inferior borders of
$D_{I}$, if they exist. More precisely, we set
\[
K^{2}=\inf \bigl\{ x^{2}\dvtx  x \mbox{ is in the support of }
\rho\bigr\}\geq0
\]
and
\[
L^{2}=\sup \bigl\{ x^{2} \dvtx x \mbox{ is in the support of }
\rho\bigr\}\leq +\infty.
\]

If $K=0$ and $L=+\infty$, then the inequality is already proven on the
set $D_{I}\setminus\{(0,0)\}$. Suppose that $K^{2}>0$. Let $y=K^{2}$
and $x \in\mathbb{R}$. We define
\[
f \dvtx  (u,v) \in\mathbb{R}^{2} \longmapsto ux+vK^{2}-
\Lambda(u,v).
\]
Denoting $c_{K}=\rho(\{K\})$, we have for all $(u,v) \in\mathbb{R}^{2}$,
\[
f(u,v)=ux-\ln\bigl(2 c_{K} \operatorname{cosh}(uK)\bigr)-\ln\int
_{\mathbb
{R}\setminus
[-K,K]}e^{uz+v(z^{2}-K^{2})}\,d\rho(z).
\]
For any $z \in\mathbb{R}\setminus\, ]{-}K,K[$, the function $v
\longmapsto\exp
(v(z^{2}-K^{2}))$ is nondecreasing. Therefore
\begin{eqnarray*}
&& \sup_{v \in\mathbb{R}} f(u,v)- \bigl(ux-\ln\bigl(2 c_{K}
\operatorname{cosh}(uK)\bigr) \bigr)
\\
&&\qquad =-\ln \biggl(\lim_{v \to-\infty}\int_{\mathbb{R}\setminus
[-K,K]}
e^{uz+v(z^{2}-K^{2})}\, d\rho(z) \biggr)=0,
\end{eqnarray*}
by the dominated convergence theorem. Indeed
\[
\forall z \in\mathbb{R}\setminus[-K,K],\forall v<-1 \qquad \bigl\llvert
e^{uz+v(z^{2}-K^{2})}\bigr\rrvert \leq e^{uz-(z^{2}-K^{2})},
\]
and the map $z \in\mathbb{R}\setminus[-K,K] \longmapsto
e^{uz-(z^{2}-K^{2})}$ is integrable with respect to~$\rho$ since it is
bounded (it is continuous and goes to~$0$ when $|z|$ goes to $+\infty
$). Hence
\[
I\bigl(x,K^{2}\bigr)=\sup_{u,v \in\mathbb{R}} f(u,v)=\sup
_{u \in\mathbb
{R}} \bigl\{ ux-\ln\bigl(2 c_{K}
\operatorname{cosh}(uK)\bigr) \bigr\}.
\]
In fact, we come back to the Bernoulli case. The reason is that, if we
condition on $T_{n}=K^{2}$ in our model, then for any $i$,
$X^{i}_{n}=-K$ or $K$.

If $c_{K}=0$, then $I(x,K^{2})=+\infty$ for any $x \neq0$, so that the
(large) inequality is verified for $y=K^{2}$. If $c_{K}>0$, then
Lemma~\ref{casbinomial} implies that, for any $\varepsilon,x$
in~$\mathbb{R}$ such
that $0\leq\varepsilon< x\leq K$,
\[
I\bigl(x,K^{2}\bigr)-I\bigl(\varepsilon,K^{2}\bigr)=
\phi_{K}(x)-\phi_{K}(\varepsilon)> \frac
{x^{2}}{2K^{2}}-
\frac{\varepsilon^{2}}{2K^{2}}.
\]

If $L<+\infty$, then we show similarly the inequality for $y=L^{2}$. Therefore
\[
\forall(x,y) \in D_{I}\setminus\bigl\{(0,0)\bigr\} \qquad 0\leq
\varepsilon<x \quad\Longrightarrow\quad I(x,y)-\frac{x^{2}}{2y}\geq I(\varepsilon,y)-
\frac
{\varepsilon
^{2}}{2y},
\]
and this inequality is strict if $(\varepsilon,y) \in\DD_{I}$. Finally we
notice that for any $y \in\mathbb{R}$, by the convexity and the
symmetry of $x
\longmapsto I(x,y)$, if $I(\varepsilon,y)=+\infty$, then for all
$x>\varepsilon$,
$I(x,y)=+\infty$. Therefore the inequality extends to each subset of
$\mathbb{R}
^{2}$ which does not contain $\mathbb{R}\times\{0\}$.
\end{pf*}

From the arguments in the previous proof, we notice that if we take
$x=0$ and $y=0$, then for any $u \in\mathbb{R}$, the function $v
\longmapsto\Lambda
(u,v)$ is nondecreasing on $\mathbb{R}$. Therefore
\[
\inf_{v \in\mathbb{R}}\Lambda(u,v)=\lim_{v \to-\infty} \Lambda
(u,v)=\lim_{v \to-\infty
} \biggl(\ln\rho\bigl(\{0\}\bigr)+\ln\int
_{\mathbb{R}\setminus\{0\}
}e^{uz+vz^{2}}\,d\rho (z) \biggr).
\]
By the dominated convergence theorem, the last integral is equal to
$\ln
\rho(\{0\})$. Hence
\[
\inf_{u,v \in\mathbb{R}^{2}} \Lambda(u,v)=\ln\rho\bigl(\{0\}\bigr).
\]
This is valid for any probability measure $\rho$ on $\mathbb{R}$.
This yields the
following lemma:

%le8 #&#
\begin{lem}\label{I(0,0)=}
If $\rho$ is a probability measure on $\mathbb{R}$, then
$I(0,0)=-\ln
\rho(\{0\})$.
\end{lem}

A consequence of Proposition~\ref{inegaliteI-F.0} and the fact that $I$
is even in its first variable is that if $D_{\Lambda}$ is an open
subset of
$\mathbb{R}^{2}$, then the function $I-F$ has a unique minimum on
$\Delta^{ *}$ at
$(0,\sigma^{2})$. Now we will extend this result to any symmetric
probability measure such that $(0,0)\in\DD_{\Lambda}$. For this we need
Mosco's theorem, which we restate next.

%de9 #&#
\begin{defi}
Let $f$ and $f_{n}$, $n \in\mathbb{N}$, be convex functions
from $\mathbb{R}^{d}$ to $[-\infty,+\infty]$. The sequence
$(f_{n})_{n \in\mathbb{N}}$
is said to Mosco converge to $f$ if for any $x \in\mathbb{R}^{d}$, we have:
\begin{longlist}[$\star$]
\item[$\star$] for each sequence $(x_{n})_{n \in\mathbb{N}}$ in
$\mathbb{R}^{d}$
converging to $x$,
\[
\liminf_{n \to+\infty}f_{n}(x_{n}) \geq f(x) ;
\]
\item[$\star$] there exists a sequence $(x_{n})_{n \in\mathbb{N}}$ in
$\mathbb{R}^{d}$
converging to $x$ and such that
\[
\limsup_{n \to+\infty}f_{n}(x_{n}) \leq f(x).
\]
\end{longlist}
\end{defi}

If $f$ is a convex function from $\mathbb{R}^{d}$ to $[-\infty
,+\infty]$, we
denote by $f^{*}$ its Fenchel--Legendre transform $f^{*}$. We have the
following theorem (see~\cite{Mosco} for a proof):

%th10 #&#
\begin{theo}[(Mosco)]\label{Mosco}
Let $f$ and $f_{n}$, $n \in\mathbb{N}$, be
functions from
$\mathbb{R}^{d}$ to $[-\infty,+\infty]$ which are convex and lower
semi-continuous. Then $(f_{n})_{n\in\mathbb{N}}$ Mosco converges to
$f$ if and
only if $(f^{*}_{n})_{n\in\mathbb{N}}$ Mosco converges to $f^{*}$.
\end{theo}

%pr11 #&#
\begin{prop}\label{corMosco}
Let $\nu$ be a probability measure on $\mathbb{R}^{d}$.
We denote
by~$L$ its log-Laplace. Let $(K_{n})_{n \in\mathbb{N}}$ be a nondecreasing
sequence of compact sets whose union is $\mathbb{R}^{d}$. For all $n
\in\mathbb{N}$,
we set $\nu_{n}=\nu(\cdot|K_{n})$ the probability~$\nu$ conditioned by
$K_{n}$, and we denote by $L_{n}$ its log-Laplace. Then $(L_{n})_{n\in
\mathbb{N}}$ Mosco converges to $L$.
\end{prop}

\begin{pf}
For\vspace*{1pt} $n$ large enough, the compact set~$K_{n}$ meets the
support of $\nu$. Thus, for $n$ large enough and $\lambda\in\mathbb
{R}^{d}$, we have
\[
L_{n}(\lambda)=\ln\int_{\mathbb{R}^{d}}e^{\langle\lambda,z
\rangle} \,d
\nu_{n}(z)=\ln \int_{K_{n}}e^{\langle\lambda,z \rangle} \,d\nu(z)
- \ln\nu (K_{n}).
\]
By the monotone convergence theorem,
\[
\lim_{n \to+\infty}L_{n}(\lambda)=\ln\int
_{\mathbb{R}^{d}}\lim_{n \to+\infty} \bigl(
\mathbh{1}_{K_{n}}(z)e^{\langle\lambda,z \rangle} \bigr) \,d\nu (z) - \lim
_{n \to
+\infty}\ln\nu(K_{n})=L(\lambda).
\]
Hence the second condition of Mosco convergence (with the $\limsup$)
is satisfied with the sequence $(\lambda_{n})_{n\in\mathbb
{N}}$ constant
equal to $\lambda$.

Let $\lambda\in\mathbb{R}^{d}$ and $(\lambda_{n})_{n\in\mathbb
{N}}$ be any sequence converging
to $\lambda$. Fatou's lemma implies that
\[
\exp L(\lambda)=\int_{\mathbb{R}^{d}}\liminf_{n \to+\infty
}
\mathbh{1}_{K_{n}}(z)e^{\langle\lambda_{n},z
\rangle}\,d\nu(z)\leq\liminf
_{n \to+\infty}\int_{\mathbb
{R}^{d}} \mathbh{1}_{K_{n}}(z)e^{\langle
\lambda_{n},z \rangle}
\,d\nu(z).
\]
Therefore
\[
L(\lambda)\leq\liminf_{n \to+\infty} \bigl(L_{n}(
\lambda_{n})+\ln \nu(K_{n}) \bigr)=\liminf
_{n \to+\infty} L_{n}(\lambda_{n}).
\]
Thus the first condition of Mosco convergence (with the $\liminf$) is verified, and the proposition is proved.
\end{pf}

%co12 #&#
\begin{cor}\label{inegaliteI-F}
If $\rho$ is a symmetric and nondegenerate probability
measure on $\mathbb{R}$, then
\[
\forall(x,y) \in\Delta^{ *}, \forall\varepsilon\in \bigl[0,|x|\bigr[ \qquad I(x,y)-
\frac{x^{2}}{2y} \geq I(\varepsilon,y)-\frac{\varepsilon
^{2}}{2y}.
\]
\end{cor}

\begin{pf}
For any $n \in\mathbb{N}$, we put $K_{n}=[-n,n]^{2}$. For $n$
large enough so that~$K_{n}$ meets the support of $\nu_{\rho}$, we define
$\nu_{n}=\nu_{\rho}(\cdot|K_{n})$, $\Lambda_{n}$ its log-Laplace and
$I_{n}$ its
Fenchel--Legendre transform. For all $(u,v) \in\mathbb{R}^{2}$,
\[
\Lambda_{n}(u,v)=\ln\int_{K_{n}}e^{ us+vt } \,
d\nu_{\rho}(s,t) - \ln\nu_{\rho
}(K_{n})\leq
\Lambda(u,v)-\ln\nu_{\rho}(K_{n}).
\]
Applying the Fenchel--Legendre transformation, we get
\[
\forall(\varepsilon,y) \in\mathbb{R}^{2}\qquad  I(\varepsilon ,y)\leq
I_{n}(\varepsilon,y)-\ln\nu _{\rho}(K_{n}).
\]
Moreover the measure $\nu_{n}$ has a bounded support, so
Proposition~\ref{inegaliteI-F.0} and the previous inequality imply that
for any $(x,\varepsilon,y)\in\mathbb{R}\times\mathbb{R}\times\,
]0,+\infty[$ such that $0\leq
\varepsilon< x$,
\[
I(\varepsilon,y)-\frac{\varepsilon^{2}}{2y}+\frac{x^{2}}{2y} \leq I_{n}(x,y)-
\ln\nu _{\rho}(K_{n}).
\]
It follows from Proposition~\ref{corMosco} that $(\Lambda_{n})_{n \in
\mathbb{N}}$
Mosco converges to $\Lambda$. Hence, by Mosco's theorem, $(I_{n})_{n
\in\mathbb{N}
}$ Mosco converges to $I$. In particular, for $(x,y) \in\mathbb
{R}^{2}$ such
that $y> 0$ and $x>\varepsilon$, there exists a sequence
$(x_{n},y_{n}) \in\mathbb{R}
^{2}$ converging to $(x,y)$ and such that
\[
\limsup_{n \to+\infty}I_{n}(x_{n},y_{n})
\leq I(x,y).
\]
Since $y> 0$ and $x>\varepsilon$, there exists $n_{0}\geq1$ such that $y_{n}>
0$ and $x_{n}> \varepsilon$ for all $n \geq n_{0}$. Therefore
\[
\forall n \geq n_{0}\qquad  I(\varepsilon,y_{n})-
\frac{\varepsilon
^{2}}{2y_{n}}+\frac
{x_{n}^{2}}{2y_{n}} \leq I_{n}(x_{n},y_{n})-
\ln\nu_{\rho}(K_{n}).
\]
Moreover $\nu_{\rho}(K_{n})$ goes to $1$ when $n$ goes to $+\infty$. Hence
\[
\limsup_{n \to+\infty}I(\varepsilon,y_{n})-
\frac{\varepsilon
^{2}}{2y}+\frac{x^{2}}{2y} \leq I(x,y).
\]
Finally $I$ is lower semi-continuous, thus
\[
\liminf_{n \to+\infty}I(\varepsilon,y_{n}) \geq I(
\varepsilon,y).
\]
This implies the announced inequality.
\end{pf}

We can now show that $I-F$ has a unique minimum on $\Delta^{ *}$ at
$(0,\sigma^{2})$:

%pr13 #&#
\begin{prop}\label{minI-F}
If $\rho$ is a symmetric probability measure on $\mathbb{R}$ with
variance $\sigma^{2}>0$ and such that $\Lambda$ is finite in a
neighborhood of
$(0,0)$, then
\[
(x,y) \in\Delta^{ *} \longmapsto I(x,y)-\frac{x^{2}}{2y}
\]
has a unique minimum at $(0,\sigma^{2})$ where it is equal to~$0$.
\end{prop}

\begin{pf}
Corollary~\ref{inegaliteI-F} implies that
\[
\forall(x,y) \in\Delta^{ *} \qquad I(x,y)-\frac{x^{2}}{2y} \geq I(0,y).
\]
Therefore $I-F$ is a nonnegative function. Since $(0,0) \in\DD_{\Lambda}$,
the function $I(0,\cdot)$ has a unique minimum at $\sigma^{2}$; see
Theorems~25.1 and~27.1~of~\cite{Rockafellar}. As a consequence, if
$I-F$ has a minimum on $\Delta^{ *}$ at $(x_{0},y_{0})$, then
$y_{0}=\sigma
^{2}$ and $I(x_{0},\sigma^{2})=x_{0}^{2}/(2\sigma^{2})$.

Moreover $(0,\sigma^{2}) \in A_{I}$, so there exists $\varepsilon>0$
such that $\mathrm{B}
_{\varepsilon}$, the open ball of radius $\varepsilon$ centered at
$(0,\sigma^{2})$, is
included in $A_{I}$. If $(x,y)$ realizes a minimum of $I-F$ on $\mathrm{B}
_{\varepsilon}$, then
\[
\bigl(u(x,y),v(x,y)\bigr)=\nabla I(x,y)=\nabla F(x,y)= \bigl(x/y,-x^{2}/
\bigl(2y^{2}\bigr) \bigr).
\]
It follows from Lemma~\ref{lemxy} that $x=0$ and thus
$u(x,y)=v(x,y)=0$. Therefore $(x,y)=(0,\sigma^{2})$. Hence
\[
\forall x \in \,]{-}\varepsilon,0[\, \cap\, ]0,\varepsilon[ \qquad I\bigl(x,
\sigma^{2}\bigr)-\frac
{x^{2}}{2\sigma^{2}}>0.
\]
Applying Corollary~\ref{inegaliteI-F} with $\varepsilon/2$, we see
that the
above inequality holds for any $x\neq0$. It follows that $x_{0}=0$.
\end{pf}

%s5 #&#
\section{Proof of Theorem~\texorpdfstring{\protect\ref{theoCVproba}}{1} with a variant of
Varadhan's lemma}
\label{ProofCVproba}

Let~$\rho$ be a symmetric probability\vspace*{-1pt} measure on $\mathbb{R}$ with positive
variance $\sigma^{2}$ and such that $(0,0) \in\DD_{\Lambda}$. The heuristics
at the end of Section~\ref{theoCV} and Proposition~\ref{minI-F} suggest
that, as $n$ goes to $+\infty$, the law of $(S_{n}/n,T_{n}/n)$ under
$\tilde{\mu}_{n,\rho}$ concentrates exponentially fast on
$(0,\sigma^2)$,
the minimum of $I-F$. Yet, in spite of the expression given in
Proposition~\ref{loiSnTn1}, we cannot apply Varadhan's lemma
(Theorem~II.7.2 of~\cite{Ellis}) directly since $\Delta^{ *}$ is not a
closed set, and $F$ is not continuous on $\Delta$.

In Section~5.5.1, we prove a
variant of Varadhan's lemma. We give the proof of Theorem~\ref
{theoCVproba} in Section~5.5.2.

%s5.1 #&#
\subsection{Around Varadhan's lemma}
\label{VaradhanVariant}

%pr14 #&#
\begin{prop}\label{TypeVaradhan}
Let $\rho$ be a probability measure on $\mathbb{R}$. We denote
by~$\tilde{\nu}_{n,\rho}$ the distribution of $(S_{n}/n,T_{n}/n)$
under $\rho^{\otimes n}$. We have
\[
\liminf_{n \to+\infty}\frac{1}{n}\ln\int_{\Delta^{ *}}
\exp \biggl(\frac{nx^{2}}{2y} \biggr) \,d\tilde{\nu}_{n,\rho}(x,y)
\geq0.
\]
Suppose that $\rho$ is nondegenerate, symmetric and that $(0,0) \in
\DD_{\Lambda}$. We assume that there exists $r>0$ such that $M_{r}+\ln
\rho(\{0\}
)<0$ with
\[
M_{r}=\sup \biggl\{ \frac{x^{2}}{2y} \dvtx (x,y)\in\mathcal{C}\cap
\mathrm{B}_{r}\setminus \bigl\{(0,0)\bigr\} \biggr\},
\]
where $\mathrm{B}_{r}$ is the open ball of radius $r$ centered at
$(0,0)$, and
$\mathcal{C}$ is the closed convex hull of $\{ (x,x^{2}) \dvtx x \mbox{ is
in the
support of } \rho \}$. If $A$ is a closed subset of~$\mathbb{R}^{2}$
which does
not contain $(0,\sigma^{2})$, then
\[
\limsup_{n \to+\infty}\frac{1}{n}\ln\int_{\Delta^{ *}\cap
A}
\exp \biggl(\frac
{nx^{2}}{2y} \biggr) \,d\tilde{\nu}_{n,\rho}(x,y)<0.
\]
\end{prop}

Let us give first some sufficient conditions to fulfill the hypothesis
of the proposition. To ensure that there exists $r>0$ such that
$M_{r}+\ln\rho(\{0\})<0$, it is enough that one of the following
conditions is satisfied:
\begin{longlist}[(a)]
\item[(a)] $\rho$ has a density.
\item[(b)] $\rho(\{0\})<1/\sqrt{e}$.
\item[(c)] There exists $c>0$ such that $\rho(]0,c[)=0$.
\item[(d)] $\rho$ is the sum of a finite number of Dirac
masses.
\end{longlist}
Indeed, the function $F$ is bounded by $1/2$ on $\mathcal
{C}\setminus
\{(0,0)\}\subset\Delta^{ *}$. Thus for any $r>0$, $M_{r}\leq1/2$.
Therefore, if $\rho$ has a density, or more generally if $\rho(\{0\}
)<e^{-1/2}$, then for all $r>0$, $M_{r}+\ln\rho(\{0\})<0$.

On the other hand, if there exists $c>0$ such that $]0,c[$ does not
intersect the support of $\rho$ (especially if $\rho$ is the sum of a
finite number of Dirac masses), then
\[
\mathcal{C}\subset\bigl\{ (x,y) \in\mathbb{R}^{2}\dvtx  c|x|\leq y \bigr
\}.
\]
Therefore
\[
\forall(x,y)\in\mathcal{C}\cap\mathrm{B}_{r}\setminus\bigl\{(0,0)
\bigr\}\qquad \frac
{x^{2}}{2y}=\frac{c|x|^{2}}{2cy}\leq\frac{|x|}{2c}\leq
\frac
{r}{2c}.
\]
Hence for any $r>0$, $M_{r}<r/2c$. Since $\rho$ is nondegenerate,
$\rho(\{
0\})<1$. Thus there exists $r>0$ such that $\ln\rho(\{0\})+r/2c<0$.
Therefore conditions (c) and (d) imply that $M_{r}+\ln\rho(\{0\})<0$.

\begin{pf*}{Proof of Proposition~\ref{TypeVaradhan}}
The large deviations principle satisfied by $(\tilde{\nu
}_{n,\rho})_{n\geq1}$ implies that
\begin{eqnarray*}
&& \liminf_{n \to+\infty}\frac{1}{n} \ln\int_{\Delta^{ *}}
\exp \biggl(\frac
{nx^{2}}{2y} \biggr) \,d\tilde{\nu}_{n,\rho}(x,y)
\\
&& \qquad \geq\liminf_{n \to+\infty}\frac{1}{n} \ln\tilde{\nu
}_{n,\rho}\bigl(\Delta^{ *}\bigr)\geq-\inf \bigl\{ I(x,y)\dvtx(x,y)
\in\taa \bigr\}=0.
\end{eqnarray*}
We prove now the second inequality. Let $\alpha>0$. The function $I$ is
lower semi-continuous on $\mathbb{R}^{2}$. Thus there exists a
neighborhood $\mathcal{U}
$ of $(0,0)$ such that
\[
\forall(x,y) \in\overline{\mathcal{U}} \qquad I(x,y) \geq \bigl(I(0,0)-\alpha \bigr)
\wedge\frac{1}{\alpha}= \bigl(-\ln\rho\bigl(\{0\}\bigr)-\alpha \bigr)\wedge
\frac
{1}{\alpha}.
\]
The above equality follows from Lemma~\ref{I(0,0)=}. By hypothesis,
there exists $r>0$ such that $M_{r}+\ln\rho(\{0\})<0$. Thus by choosing
$\alpha$ sufficiently small, we can assume that
\[
M_{r}+\ln\rho\bigl(\{0\}\bigr)+\alpha<0 \quad \mbox{and} \quad M_{r}-
\frac
{1}{\alpha}<0.
\]
Since $M_{r}$ decreases with $r$, we can take $r$ small enough so that
$\mathrm{B}_{r} \subset\mathcal{U}$. Notice next that
$(S_{n}/n,T_{n}/n) \in\mathcal{C}$
almost surely. Therefore, setting $\mathcal{C}^{*}=\mathcal
{C}\setminus\{(0,0)\}$,
\[
\int_{\Delta^{ *}\cap A}\exp \biggl(\frac{nx^{2}}{2y} \biggr) \,d
\tilde{\nu }_{n,\rho}(x,y)=\int_{\mathcal{C}^{*}\cap A}\exp \biggl(
\frac
{nx^{2}}{2y} \biggr) \,d\tilde{\nu}_{n,\rho}(x,y).
\]
Let us decompose
\[
\mathcal{C}^{*}\cap A \subset \bigl(\mathcal{C}^{*}\cap
\mathrm {B}_{r} \bigr)\cup \bigl(\mathcal{C} \cap\mathrm{B}_{r}^{c}
\cap A \bigr).
\]
We have
\[
\int_{\mathcal{C}^{*}\cap\mathrm{B}_{r}} \exp \biggl(\frac
{nx^{2}}{2y} \biggr) \,d
\tilde{\nu}_{n,\rho}(x,y)\leq\exp(nM_{r}) \tilde{\nu
}_{n,\rho
} (\mathcal{U}).
\]
The large deviation principle satisfied by $(\tilde{\nu}_{n,\rho
})_{n\geq1}$ implies that
\begin{eqnarray*}
&&\limsup_{n \to+\infty}\frac{1}{n}\ln\int_{\mathcal{C}^{*}\cap
\mathrm{B}_{r}}
\exp \biggl(\frac
{nx^{2}}{2y} \biggr) \,d\tilde{\nu}_{n,\rho}(x,y)
\\
&& \qquad \leq M_{r}-\inf_{\overline{\mathcal{U}}}I\leq \bigl(M_{r}+
\ln\rho \bigl(\{0\}\bigr)+\alpha \bigr)\vee \biggl( M_{r}-
\frac{1}{\alpha} \biggr).
\end{eqnarray*}
Next, the set $\mathcal{C}\cap\mathrm{B}_{r}^{c} \cap A$ is closed
and does not
contain $(0,0)$. Thus the function~$F$ is continuous on this set.
Moreover $F$ is bounded on $\mathcal{C}^{*}$. Hence Lemma~\ref{VaradhanUpper}
in the \hyperref[appA]{Appendix} and Lemma~1.2.15 of~\cite{DZ} imply that
\begin{eqnarray*}
&& \limsup_{n \to+\infty}\frac{1}{n}\ln\int_{\mathcal{C}^{*}\cap
A}
\exp \biggl(\frac
{nx^{2}}{2y} \biggr) \,d\tilde{\nu}_{n,\rho}(x,y)
\\
&& \qquad \leq\max \biggl(M_{r}+\ln\rho\bigl(\{0\}\bigr)+\alpha,
M_{r}-\frac
{1}{\alpha}, \sup_{\mathcal{C}\cap\mathrm{B}_{r}^{c} \cap A} (F-I) \biggr).
\end{eqnarray*}
Since $\rho$ is symmetric and $(0,0)\in\DD_{\Lambda}$, Proposition~\ref
{minI-F} implies that $G=I-F$ has a unique minimum at $(0,\sigma^{2})$ on
$\Delta^{ *}$.
Suppose that the infimum of $G$ over $\mathcal{C}\cap\mathrm
{B}_{r}^{c}\cap A$ is
null. Then there exists a sequence $(x_{k},y_{k})_{k \in\mathbb{N}}$
in $\mathcal{C}
\cap\mathrm{B}_{r}^{c} \cap A\subset\Delta^{ *}$ such that
\[
\lim_{k \to+\infty}G(x_{k},y_{k})=\inf
_{\mathcal{C}\cap\mathrm{B}_{r}^{c} \cap
A}G=0.
\]
For $k$ large enough,\vspace*{1pt} $G(x_{k},y_{k}) \leq1/2$. Thus $I(x_{k},y_{k})
\leq1$l; that is, $(x_{k},y_{k})$ belongs to the compact set
$\{(u,v) \in\mathbb{R}^{2} \dvtx I(u,v)\leq1 \}$. Up to the extraction of a
subsequence, we suppose that $(x_{k},y_{k})_{k \in\mathbb{N}}$
converges to
some $(x_{0},y_{0})$, which belongs to the closed subset $\mathcal
{C}\cap\mathrm{B}
_{r}^{c} \cap A$. Moreover $G$ is lower semi-continuous, and hence
\[
0=\limsup_{k \to+\infty} G(x_{k},y_{k}) \geq
G(x_{0},y_{0})\geq0.
\]
Therefore $G(x_{0},y_{0})=0$, and thus $(x_{0},y_{0})=(0,\sigma^{2})
\in\mathcal{C}\cap\mathrm{B}_{r}^{c} \cap A$, which is absurd since $A$ does not contain
$(0,\sigma^{2})$. Thus the infimum of $G$ over $\mathcal{C}\cap
\mathrm{B}_{r}^{c}\cap A$
is positive. Therefore
\[
\max \biggl(M_{r}+\ln\rho\bigl(\{0\}\bigr)+\alpha, M_{r}-
\frac{1}{\alpha
}, \sup_{\mathcal{C}
\cap\mathrm{B}_{r}^{c} \cap A}(F-I) \biggr)<0.
\]
This proves the second inequality.
\end{pf*}

%s5.2 #&#
\subsection{Proof of Theorem~\texorpdfstring{\protect\ref{theoCVproba}}{1}}
\label{subProofCVproba}

Let $\rho$ be a symmetric probability measure on $\mathbb{R}$ with positive
variance $\sigma^{2}$ and such that
\[
\exists v_0>0 \qquad \int_{\mathbb{R}}e^{v_{0}z^{2}} \,d\rho
(z)<+\infty.
\]
This implies that $\mathbb{R}\times\,]{-}\infty,v_{0}[ \, \subset
D_{\Lambda}$ and thus
$(0,0) \in\DD_{\Lambda}$. We
assume that one of the four conditions given
in the paragraph below Proposition~\ref{TypeVaradhan} is satisfied.

We denote by $\theta_{n,\rho}$ the distribution of $(S_{n}/n,T_{n}/n)$
under $\tilde{\mu}_{n,\rho}$. Let $U$ be an open neighborhood of
$(0,\sigma^{2})$ in $\mathbb{R}^{2}$. Propositions~\ref{loiSnTn1}
and~\ref{TypeVaradhan} imply that
\begin{eqnarray*}
\limsup_{n \to+\infty}\frac{1}{n}\ln\theta_{n,\rho
}
\bigl(U^{c}\bigr)&=&\limsup_{n \to+\infty}\frac{1}{n}
\ln \int_{\Delta^{ *}\cap U^{c}}\exp \biggl(\frac{nx^{2}}{2y} \biggr) \,d
\tilde{\nu}_{n,\rho}(x,y)
\\
&&{}-\liminf_{n \to+\infty}\frac{1}{n}\ln\int
_{\Delta
^{ *}}\exp \biggl(\frac
{nx^{2}}{2y} \biggr)\,d\tilde{\nu}_{n,\rho}(x,y)<0.
\end{eqnarray*}
Hence there exist $\varepsilon>0$ and $n_{0}\in\mathbb{N}$ such that
for any $n>n_{0}$,
\[
\theta_{n,\rho}\bigl(U^{c}\bigr) \leq e^{-n\varepsilon} \mathop{\longrightarrow}_{n \to\infty} 0.
\]
Thus, for each open neighborhood $U$ of $(0,\sigma^{2})$,
\[
\lim_{n \to+\infty}\tilde{\mu}_{n,\rho} \biggl( \biggl(
\frac
{S_{n}}{n},\frac{T_{n}}{n} \biggr)\in U^{c} \biggr)=0.
\]
This means that, under $\tilde{\mu}_{n,\rho}$, $(S_{n}/n,T_{n}/n)$
converges in probability to $(0,\sigma^{2})$. This completes the proof of
Theorem~\ref{theoCVproba}.

%s6 #&#
\section{Expansion of $I-F$ around its minimum}
\label{Expansion}

In this section, which may be omitted on a first reading, we compute
the expansion of the function $I-F$ around $(0,\sigma^2)$, its minimum over
$\Delta^{ *}$. These computations are crucial because they explain
why the
fluctuations in Theorem~\ref{theoFluctuations} are of order $n^{3/4}$,
and they give us the term in the exponential in the limiting law.

If $\rho$ is a symmetric probability measure whose support contains at
least three points and if $(0,0)\in\DD_{L}$, then $(0,\sigma^{2})=\nabla
\Lambda(0,0) \in\nabla\Lambda(\DD_{\Lambda})=A_{I}$, the admissible domain of $I$.
Proposition~\ref{Dadmissible} in the \hyperref[appA]{Appendix} implies that $I$ is
$\mathrm{C}^{\infty}$ in the neighborhood of $(0,\sigma^{2})$ and that
\begin{eqnarray*}
\nabla I\bigl(0,\sigma^{2}\bigr)&=& \bigl(u\bigl(0,\sigma^{2}
\bigr),v\bigl(0,\sigma^{2}\bigr)\bigr)=(\nabla \Lambda)^{-1}
\bigl(0,\sigma ^{2}\bigr)=(0,0),
\\
\mathrm{D}^{2}_{(0,\sigma^{2})}I &=&  \bigl(\mathrm {D}^{2}_{(0,0)}
\Lambda \bigr)^{-1}= %
\pmatrix{ \sigma^{2} & 0
\vspace*{2pt}
\cr
0 & \mu_{4}-\sigma^{4}}^{-1}=
\pmatrix{ 1/\sigma^{2} & 0 \vspace*{2pt}
\cr
0 & 1/\bigl(
\mu_{4}-\sigma^{4}\bigr)},
\end{eqnarray*}
since $\mathrm{D}^{2}_{(0,0)}\Lambda$ is the covariance
matrix of
$\nu_{\rho}$. Hence, up to the second order, the expansion of $I-F$ in
the neighborhood of $(0,\sigma^{2})$ is
\[
I(x,y)-F(x,y)=\frac{(y-\sigma^{2})^{2}}{2(\mu_{4}-\sigma^{4})}+o\bigl(\bigl\| x,y-\sigma^{2}\bigr\|
^{2}\bigr).
\]
We need to push further the expansion of $I-F$.

Consider the case of the Gaussian $\mathcal{N}(0,\sigma^{2})$. We can
explicitly
compute $I$ in the following way:
\[
\forall(x,y)\in\Delta^{ *} \qquad I(x,y)=\frac{1}{2} \biggl(
\frac
{y}{\sigma
^{2}}-1-\ln \biggl(\frac{y-x^{2}}{\sigma^{2}} \biggr) \biggr).
\]
In the neighborhood of $(0,\sigma^{2})$, we have
\[
I(x,y)-F(x,y)\sim\frac{x^{4}}{4\sigma^{4}}+\frac{(y-\sigma
^{2})^{2}}{4\sigma^{2}}.
\]
In fact, we have a similar expansion in a more general case:

%pr15 #&#
\begin{prop}\label{DLdeI-F}
If $\rho$ is\vspace*{-1.5pt} a symmetric probability measure on $\mathbb{R}$ whose
support contains at least three points and such that $(0,0) \in\DD_{\Lambda}$, then $I$ is $\mathrm{C}^{\infty}$ in the neighborhood of
$(0,\sigma^{2})$.
If
$\mu_{4}$ denotes the fourth moment of $\rho$, then when $(x,y)$ goes to
$(0,\sigma^{2})$,
\[
I(x,y)-\frac{x^{2}}{2y}\sim\frac{(y-\sigma^{2})^{2}}{2(\mu
_{4}-\sigma
^{4})}+\frac{\mu_{4}x^{4}}{12\sigma^{8}}.
\]
\end{prop}

\begin{pf}
If $(0,0) \in\DD_{\Lambda}$,
then $(0,\sigma^{2})=\nabla\Lambda(0,0)
\in\nabla\Lambda(\DD_{\Lambda
})=A_{I}$, and Proposition~\ref{Dadmissible} in
the \hyperref[appA]{Appendix} implies that the function $I$ is $\mathrm{C}^{\infty}$
on $A_{I}$.
Moreover, if we denote the inverse function of $\nabla\Lambda$ by
$(x,y)\longmapsto(u(x,y),v(x,y))$, then, for all $(x,y) \in A_{I}$,
\[
\nabla I(x,y)=\bigl(u(x,y),v(x,y)\bigr)\quad \mbox{and}\quad \mathrm{D}^{2}_{(x,y)}I=
\bigl(\mathrm{D}^{2}_{(u(x,y),v(x,y))}\Lambda \bigr)^{-1}.
\]
The hypothesis $(0,0) \in\DD_{\Lambda}$ also implies that $\rho$ has finite
moments of all orders. The expansion of $F$ to the fourth order in the
neighborhood of $(0,\sigma^{2})$ is
\[
F(x,y)=\frac{x^{2}}{2\sigma^{2}}-\frac{x^{2}(y-\sigma^{2})}{2\sigma
^{4}}+\frac
{x^{2}(y-\sigma^{2})^{2}}{2\sigma^{6}}+o\bigl(\bigl\|x,y-
\sigma^{2}\bigr\|^{4}\bigr).
\]
Therefore, in the neighborhood of $(0,0)$,
\begin{eqnarray*}
&& I\bigl(x,h+\sigma^{2}\bigr)-F\bigl(x,h+\sigma^{2}\bigr)\\
&&\qquad=
\frac{h^{2}}{2(\mu_{4}-\sigma
^{4})}+a_{3,0}x^{3}+a_{2,1}x^{2}h+a_{1,2}xh^{2}+a_{0,3}h^{3}
\\
&&\qquad\quad{}+a_{4,0}x^{4}+a_{3,1}x^{3}h+a_{2,2}x^{2}h^{2}+a_{1,3}xh^{3}+a_{0,4}h^{4}+o
\bigl(\| x,h\|^{4}\bigr),
\end{eqnarray*}
with, for any $(i,j) \in\mathbb{N}$ such that $i+j\in\{3,4\}$,
\[
a_{i,j}=\frac{1}{i!j!}\frac{\partial^{i+j}I}{\partial x^{i}\,\partial
y^{j}}\bigl(0,
\sigma^{2}\bigr),
\]
except for
\[
a_{2,1}=\frac{1}{2} \biggl(\frac{\partial^{3}I}{\partial
x^{2}\,\partial
y}\bigl(0,
\sigma^{2}\bigr)+\frac{1}{\sigma^{4}} \biggr) \quad\mbox{and}\quad
a_{2,2}=\frac{1}{4}\frac{\partial^{4}I}{\partial x^{2}\,\partial
y^{2}}\bigl(0,
\sigma^{2}\bigr)-\frac{1}{2\sigma^{6}}.
\]
If we prove that $a_{4,0}>0$, then the terms $xh^{2}$, $h^{3}$,
$x^{3}h$, $x^{2}h^{2}$, $xh^{3}$ and $h^{4}$ are negligible compared to
$a_{4,0}x^{4}+a_{0,2}h^{2}$ when $(x,h)$ goes to $(0,0)$.
Next, the symmetry of $I-F$ in the first variable implies that
$a_{3,0}=0$. If we show that $a_{2,1}=0$, then when $(x,y)$ goes to
$(0,\sigma^{2})$,
\[
I(x,y)-F(x,y)= \biggl(\frac{(y-\sigma^{2})^{2}}{2(\mu_{4}-\sigma
^{4})}+a_{4,0}x^{4}
\biggr) \bigl(1+o(1)\bigr).
\]
To conclude it is enough to show that $a_{2,1}=0$ and $a_{4,0}=\mu
_{4}/(12\sigma^{8})$, that is,
\[
\frac{\partial^{3}I}{\partial x^{2}\,\partial y}\bigl(0,\sigma^{2}\bigr)=-\frac
{1}{\sigma
^{4}}\quad
\mbox{and}\quad \frac{\partial^{4}I}{\partial
x^{4}}\bigl(0,\sigma ^{8}\bigr)=
\frac{2\mu_{4}}{\sigma^{2}}.
\]
For any $j \in\mathbb{N}$, we introduce the function $f_{j}$ defined
on $\DD_{\Lambda}$ by
\[
\forall(u,v) \in\DD_{\Lambda}\qquad f_{j}(u,v)=
\int_{\mathbb{R}
}x^{j}e^{ux+vx^{2}} \,d\rho(x) \biggl(
\int_{\mathbb
{R}}e^{ux+vx^{2}} \,d\rho(x) \biggr)^{-1}.
\]
These functions are $\mathrm{C}^{\infty}$ on $\DD_{\Lambda}$, and they verify the
following properties:
\begin{longlist}[$\star$]
\item[$\star$]  $f_{0}$ is the identity function on $\mathbb{R}^{2}$ and
\[
f_{1}=\frac{\partial\Lambda}{\partial u} \quad\mbox{and}\quad f_{2}=
\frac{\partial\Lambda}{\partial v}.
\]
\item[$\star$]  For all $j \in\mathbb{N}$, $f_{j}(0,0)=\mu_{j}$ is the
$j$th moment of
$\rho$. It is null if $j$ is odd, since $\rho$ is symmetric.
Moreover, for
any $j\in\mathbb{N}$,
\[
\frac{\partial f_{j}}{\partial u}=f_{j+1}-f_{j}f_{1}\quad \mbox{and}\quad
\frac{\partial f_{j}}{\partial v}=f_{j+2}-f_{j}f_{2}.
\]
Therefore, for all $(x,y) \in A_{I}$,
\begin{eqnarray*}
\mathrm{D}^{2}_{(x,y)}I&= & \bigl(\mathrm {D}^{2}_{(u(x,y),v(x,y))}
\Lambda \bigr)^{-1}
\\
&= &
\pmatrix{ f_{2}-f_{1}^{2} &
f_{3}-f_{1}f_{2} \vspace*{2pt}
\cr
f_{3}-f_{1}f_{2} & f_{4}-f_{2}^{2}}^{-1}
\bigl(u(x,y),v(x,y)\bigr).
\end{eqnarray*}
Denoting by
$g=(f_{2}-f_{1}^{2})(f_{4}-f_{2}^{2})-(f_{3}-f_{1}f_{2})^{2}$, the
determinant of the positive definite symmetric matrix $\mathrm
{D}^{2}\Lambda
$, we get that for any $(x,y) \in A_{I}$,
\[
\mathrm{D}^{2}_{(x,y)}I=\frac{1}{g(u(x,y),v(x,y))} \pmatrix{
f_{4}-f_{2}^{2} & f_{1}f_{2}-f_{3}
\vspace*{2pt}
\cr
f_{1}f_{2}-f_{3} &
f_{2}-f_{1}^{2}} \bigl(u(x,y),v(x,y)\bigr).
\]
Moreover $(u(0,\sigma^{2}),v(0,\sigma^{2}))=(0,0)$ thus
\begin{eqnarray*}
\frac{\partial u}{\partial x}\bigl(0,\sigma^{2}\bigr) &=& \frac{\partial^{2}
I}{\partial
x^{2}}\bigl(0,
\sigma^{2}\bigr)=\frac{f_{4}-f_{2}^{2}}{g}(0,0)=\frac{\mu
_{4}-\sigma^{4}}{\sigma
^{2}(\mu_{4}-\sigma^{4})}=
\frac{1}{\sigma^{2}},
\\
\frac{\partial v}{\partial y}\bigl(0,\sigma^{2}\bigr)&=& \frac{\partial^{2}
I}{\partial
y^{2}}
\bigl(0,\sigma^{2}\bigr)=\frac{f_{2}-f_{1}^{2}}{g}(0,0)=\frac{\sigma
^{2}}{\sigma^{2}(\mu
_{4}-\sigma^{4})}=
\frac{1}{\mu_{4}-\sigma^{4}},
\\
\frac{\partial u}{\partial y}\bigl(0,\sigma^{2}\bigr)&=&\frac{\partial
v}{\partial
x}
\bigl(0,\sigma^{2}\bigr)=\frac{\partial^{2} I}{\partial x\, \partial
y}\bigl(0,\sigma
^{2}\bigr)=\frac{f_{1}f_{2}-f_{3}}{g}(0,0)=0.
\end{eqnarray*}
Differentiating with respect to $y$, we get
\[
\frac{\partial^{3} I}{\partial y\,\partial x^{2}}= \frac{\partial
u}{\partial y}\times\frac{\partial}{\partial u} \biggl(
\frac
{f_{4}-f_{2}^{2}}{g} \biggr) (u,v)+\frac{\partial v}{\partial y}\times \frac{\partial}{\partial v}
\biggl(\frac{f_{4}-f_{2}^{2}}{g} \biggr) (u,v).
\]
The first term of the addition, taken at $(0,\sigma^{2})$, is null.
For the
second term, we need to compute the partial derivative of
$(f_{4}-f_{2}^{2})/g$ with respect to $v$,
\begin{eqnarray*}
\frac{\partial}{\partial v} \biggl(\frac{f_{4}-f_{2}^{2}}{g} \biggr)&=&\frac{1}{g}
\times\frac{\partial}{\partial v} \bigl(f_{4}-f_{2}^{2}
\bigr)-\frac{f_{4}-f_{2}^{2}}{g^{2}}\times\frac
{\partial g}{\partial v}
\\
&=&\frac{f_{6}-3f_{2}f_{4}+2f_{2}^{3}}{g}-\frac
{f_{4}-f_{2}^{2}}{g^{2}}\times\frac{\partial g}{\partial v}.
\end{eqnarray*}
Let us differentiate with respect to $v$,
\begin{eqnarray*}
\frac{\partial g}{\partial
v}&=& f_{2}(f_{6}-f_{4}f_{2})+f_{4}
\bigl(f_{4}-f_{2}^{2}\bigr)-f_{1}^{2}(f_{6}-f_{4}f_{2})
\\
&&{}-2f_{4}f_{1}(f_{3}-f_{1}f_{2})-3f_{2}^{2}
\bigl(f_{4}-f_{2}^{2}\bigr)-2f_{3}(f_{5}-f_{3}f_{2})
\\
&&{}+2f_{1}f_{2}(f_{5}-f_{3}f_{2})+2f_{2}f_{3}(f_{3}-f_{1}f_{2})+2f_{1}f_{3}
\bigl(f_{4}-f_{2}^{2}\bigr).
\end{eqnarray*}
Taken at $(0,0)$, each term with even subscript vanishes and we have
\begin{eqnarray*}
\frac{\partial g}{\partial v}(0,0)&=&\sigma^{2}\bigl(\mu_{6}-\mu
_{4}\sigma^{2}\bigr)+\mu _{4}\bigl(
\mu_{4}-\sigma^{4}\bigr)-3\sigma^{4}\bigl(
\mu_{4}-\sigma^{4}\bigr)
\\
&=& \sigma^{2}\mu_{6}-3\mu_{4}
\sigma^{4}+2\sigma^{8}+\bigl(\mu _{4}-
\sigma^{4}\bigr)^{2}.
\end{eqnarray*}
Finally\vspace*{-2pt}
\begin{eqnarray*}
&&
\frac{\partial}{\partial v} \biggl(\frac{f_{4}-f_{2}^{2}}{g} \biggr) (0,0)
\\[-2pt]
&& \qquad =\frac{\mu_{6}-3\sigma^{2}\mu_{4}+2\sigma^{6}}{\sigma^{2}(\mu
_{4}-\sigma^{4})}-\frac{\sigma
^{2}\mu_{6}-3\mu_{4}\sigma^{4}+2\sigma^{8}+(\mu_{4}-\sigma
^{4})^{2}}{\sigma^{4}(\mu
_{4}-\sigma^{4})},
\end{eqnarray*}
which is equal to $(\sigma^{4}-\mu_{4})/\sigma^4$ after
simplification. Therefore
\begin{eqnarray*}
\frac{\partial^{3} I}{\partial y\,\partial x^{2}}\bigl(0,\sigma ^{2}\bigr) &=& 0+\frac
{\partial v}{\partial y}
\bigl(0,\sigma^{2}\bigr) \frac{\partial}{\partial
v} \biggl(\frac{f_{4}-f_{2}^{2}}{g}
\biggr) (0,0)
\\[-2pt]
&=& \frac{1}{\mu_{4}-\sigma^{4}}\times\frac{\sigma^{4}-\mu
_{4}}{\sigma^{4}}=-\frac
{1}{\sigma^{4}}.
\end{eqnarray*}
This is what we wanted to prove. Let us compute now the fourth partial
derivative of $I$ with respect to $x$. We have to obtain first an
expression of the third partial derivative of $I$ with respect to $x$,
\[
\frac{\partial^{3} I}{\partial x^{3}}= \frac{\partial u}{\partial
x}\times\frac{\partial}{\partial u} \biggl(
\frac
{f_{4}-f_{2}^{2}}{g} \biggr) (u,v)+\frac{\partial v}{\partial x}\times \frac{\partial}{\partial v}
\biggl(\frac{f_{4}-f_{2}^{2}}{g} \biggr) (u,v).
\]
The only term we do not know is the partial derivative with respect to
$u$ of $(f_{4}-f_{2}^{2})/g$. We have
\begin{eqnarray*}
\frac{\partial}{\partial u} \biggl(\frac{f_{4}-f_{2}^{2}}{g} \biggr) &=& \frac{1}{g}
\times\frac{\partial}{\partial u} \bigl(f_{4}-f_{2}^{2}
\bigr)-\frac{f_{4}-f_{2}^{2}}{g^{2}}\times\frac
{\partial g}{\partial u}
\\
&=&\frac{f_{5}-f_{4}f_{1}-2f_{2}f_{3}+2f_{2}^{2}f_{1}}{g}-\frac
{f_{4}-f_{2}^{2}}{g^{2}}\times\frac{\partial g}{\partial
u},
\end{eqnarray*}
with\vspace*{-2pt}
\begin{eqnarray*}
\frac{\partial g}{\partial
u} &=& f_{2}(f_{5}-f_{4}f_{1})+f_{4}(f_{3}-f_{2}f_{1})-f_{1}^{2}(f_{5}-f_{4}f_{1})
\\[-2pt]
&&{}-2f_{4}f_{1}\bigl(f_{2}-f_{1}^{2}
\bigr)-3f_{2}^{2}(f_{3}-f_{2}f_{1})-2f_{3}(f_{4}-f_{3}f_{1})
\\[-2pt]
&&{}+2f_{1}f_{2}(f_{4}-f_{3}f_{1})+2f_{2}f_{3}
\bigl(f_{2}-f_{1}^{2}\bigr)+2f_{1}f_{3}(f_{3}-f_{2}f_{1}).
\end{eqnarray*}
Notice that this quantity vanishes at $(0,0)$. Therefore the partial
derivative of $(f_{4}-f_{2}^{2})/g$ with respect to $u$, taken at
$(0,0)$, is null as well, and we get back that the third partial
derivative of $I$ with respect to $x$, taken at $(0,\sigma^{2})$, is null.
Differentiating once more, we obtain
\begin{eqnarray*}
\frac{\partial^{4}I}{\partial x^{4}} &=& \frac{\partial^{2}
u}{\partial x^{2}}\times\frac{\partial}{\partial u} \biggl(
\frac
{f_{4}-f_{2}^{2}}{g} \biggr) (u,v)+\frac{\partial^{2} v}{\partial
x^{2}}\times\frac{\partial}{\partial v}
\biggl(\frac{f_{4}-f_{2}^{2}}{g} \biggr) (u,v)
\\[-1pt]
&&{}+\frac{\partial u}{\partial x}\times \biggl(\frac
{\partial
u}{\partial x}\times\frac{\partial^{2} }{\partial u^{2}}
\biggl(\frac
{f_{4}-f_{2}^{2}}{g} \biggr) (u,v)+\frac{\partial v}{\partial x}\times
\frac{\partial^{2} }{\partial v \,\partial u} \biggl(\frac
{f_{4}-f_{2}^{2}}{g} \biggr) (u,v) \biggr)
\\[-1pt]
&&{}+\frac{\partial v}{\partial x}\times \biggl(\frac
{\partial
u}{\partial x}\times\frac{\partial^{2} }{\partial u \,\partial v}
\biggl(\frac{f_{4}-f_{2}^{2}}{g} \biggr) (u,v)+\frac{\partial v}{\partial
x}\times
\frac{\partial^{2} }{\partial v^{2}} \biggl(\frac
{f_{4}-f_{2}^{2}}{g} \biggr) (u,v) \biggr).
\end{eqnarray*}
Let us compute it at $(0,\sigma^{2})$,
\begin{eqnarray*}
\frac{\partial^{4} I}{\partial x^{4}}\bigl(0,\sigma^{2}\bigr)&=& \frac{1}{\sigma
^{2}}
\biggl(\frac{1}{\sigma^{2}}\frac{\partial^{2} }{\partial u^{2}} \biggl(\frac
{f_{4}-f_{2}^{2}}{g} \biggr)
(0,0)+0 \biggr)
\\
&&{}+0+\frac{\sigma^{4}-\mu_{4}}{\sigma^{4}}\frac{\partial^{2}
v}{\partial
x^{2}}\bigl(0,\sigma^{2}\bigr)+0,
\end{eqnarray*}
with
\[
\frac{\partial^{2} v}{\partial x^{2}}\bigl(0,\sigma^{2}\bigr)=\frac{\partial
}{\partial x}
\biggl(\frac{\partial^{2} I}{\partial x \,\partial
y} \biggr) \bigl(0,\sigma^{2}\bigr)=
\frac{\partial^{3} I}{\partial x^{2}\, \partial
y}\bigl(0,\sigma ^{2}\bigr)=-\frac{1}{\sigma^{4}}
\]
and
\begin{eqnarray*}
\frac{\partial^{2} }{\partial u^{2}} \biggl(\frac
{f_{4}-f_{2}^{2}}{g} \biggr)&=& \frac{1}{g}
\frac{\partial
^{2}}{\partial
u^{2}}\bigl(f_{4}-f_{2}^{2}\bigr)-
\frac{2}{g^{2}}\frac{\partial g }{\partial
u} \frac{\partial}{\partial u}\bigl(f_{4}-f_{2}^{2}
\bigr)
\\
&&{}-\frac{f_{4}-f_{2}^{2}}{g^{2}}\frac{\partial^{2} g }{\partial
u^{2}}+\frac{2}{g^{3}} \biggl(
\frac{\partial g }{\partial u} \biggr)^{2}\bigl(f_{4}-f_{2}^{2}
\bigr).
\end{eqnarray*}
Hence
\[
\frac{\partial^{2} }{\partial u^{2}} \biggl(\frac
{f_{4}-f_{2}^{2}}{g} \biggr) (0,0)=\frac{1}{\sigma^{4}(\mu_{4}-\sigma
^{4})}
\biggl(\sigma^{2}\frac{\partial^{2}}{\partial
u^{2}}\bigl(f_{4}-f_{2}^{2}
\bigr) (0,0)-\frac
{\partial^{2} g }{\partial u^{2}}(0,0) \biggr).
\]
The two remaining terms are the derivatives of quantities which we have
already computed. In the following, we evaluate them directly at
$(0,0)$, which is straightforward since $f_{j}(0,0)=0$ when $j$ is odd:
\[
\frac{\partial^{2}}{\partial u^{2}}\bigl(f_{4}-f_{2}^{2}\bigr)
(0,0)=\frac
{\partial
}{\partial u}\bigl(f_{5}-f_{4}f_{1}-2f_{2}f_{3}+2f_{2}^{2}f_{1}
\bigr) (0,0)=\mu _{6}-3\sigma^{2}\mu_{4}+2
\sigma^{6}
\]
and
\begin{eqnarray*}
\frac{\partial^{2} g }{\partial u^{2}}(0,0) &=& \frac{\partial
}{\partial
u} \biggl(\frac{\partial g }{\partial u} \biggr)
(0,0)=\sigma^{2}\bigl(\mu _{6}-\mu _{4}
\sigma^{2}\bigr)+\mu_{4}\bigl(\mu_{4}-
\sigma^{4}\bigr)-0-2\mu_{4}\sigma^{4}
\\
&&{}-3\sigma^{4}\bigl(\mu_{4}-\sigma^{4}\bigr)-2
\mu_{4}^{2}+2\sigma^{4}\mu _{4}+2
\sigma^{4}\mu_{4}+0.
\end{eqnarray*}
This is equal to $\sigma^{2}\mu_{6}-\mu_{4}^{2}+3\sigma^{8}-3\mu
_{4}\sigma^{4}$
after simplification. Thus we have
\begin{eqnarray*}
\frac{\partial^{2} }{\partial u^{2}} \biggl(\frac
{f_{4}-f_{2}^{2}}{g} \biggr) (0,0)&=&
\frac{\sigma^{2}\mu_{6}-3\sigma
^{4}\mu_{4}+2\sigma
^{8}-\sigma^{2}\mu_{6}+\mu_{4}^{2}-3\sigma^{8}+3\mu_{4}\sigma
^{4}}{\sigma^{4}(\mu_{4}-\sigma
^{4})}
\\
&=&\frac{\mu_{4}^{2}-\sigma^{8}}{\sigma^{4}(\mu_{4}-\sigma
^{4})}=\frac{\mu_{4}+\sigma
^{4}}{\sigma^{4}}.
\end{eqnarray*}
Finally
\[
\frac{\partial^{2} I}{\partial x^{4}}\bigl(0,\sigma^{2}\bigr)=\frac{\mu
_{4}+\sigma^{4}}{\sigma
^{8}}-
\frac{\sigma^{4}-\mu_{4}}{\sigma^{8}}=\frac{2\mu
_{4}}{\sigma^{8}}.
\]
\end{longlist}
We obtain the announced term and the proof is complete.
\end{pf}

%s7 #&#
\section{Proof of Theorem~\texorpdfstring{\protect\ref{theoFluctuations}}{2}}
\label{ProofFluctuations}

We first give conditions on the probability measure $\rho$ in order to
apply Theorem~\ref{exp(-nI)} (see Appendix~\ref{appA}) to the distribution~$\nu
_{\rho}$. We will use then Laplace's method, as we announced in the
heuristics of Section~\ref{theoCV}, to obtain the fluctuations
Theorem~\ref{theoFluctuations}. The proof will rely on the expansion of
$I-F$ around $(0,\sigma^{2})$ given in Proposition~\ref{DLdeI-F}. We will
also use the variant of Varadhan's lemma, stated in Proposition~\ref
{TypeVaradhan}. We start with the following lemma:

%le16 #&#
\begin{lem}
If $\rho$ has a probability density $f$ with respect to the
Lebesgue measure on $\mathbb{R}$, then $\nu_{\rho}^{*2}$ has the density
\[
f_{2}\dvtx  (x,y)\longmapsto\frac{1}{\sqrt{2y-x^{2}}} f \biggl(
\frac
{x+\sqrt
{2y-x^{2}}}{2} \biggr) f \biggl(\frac{x-\sqrt{2y-x^{2}}}{2} \biggr)
\mathbh{1}_{x^{2}<2y}
\]
with respect to the Lebesgue measure on $\mathbb{R}^{2}$.
\end{lem}

\begin{pf}
Let $h$ be a bounded continuous function from $\mathbb{R}^{2}$ to
$\mathbb{R}$. We have
\begin{eqnarray*}
\int_{\mathbb{R}^{2}}h(x,y) \,d\nu_{\rho}^{*2}(x,y)
&=&
\int_{\mathbb{R}
^{2}}h\bigl(\bigl(z,z^{2}\bigr)+
\bigl(t,t^{2}\bigr)\bigr) \,d\rho(z) \,d\rho(t)
\\
&=& \int_{D^{+}}h\bigl(z+t,z^{2}+t^{2}
\bigr)f(z)f(t) \,dz \,dt\\
&&{}+\int_{D^{-}}h\bigl(z+t,z^{2}+t^{2}
\bigr)f(z)f(t) \,dz \,dt,
\end{eqnarray*}
with $D^{+}=\{ (z,t)\in\mathbb{R}^{2}\dvtx z> t \}$ and $D^{-}=\{
(z,t)\in\mathbb{R}
^{2}\dvtx z< t \}$. Indeed, the Lebesgue measure of the set $\{ (z,t)\in
\mathbb{R}
^{2}\dvtx z=t \}$ is null. Let us denote, respectively, by $I_{+}$ and
$I_{-}$ the two previous integrals.

We define $\phi\dvtx  (z,t) \in\mathbb{R}^{2} \longmapsto
(u,v)=(z+t,z^{2}+t^{2})$. It is a one\vspace*{1pt} to one map from $D^{+}$ (resp.,
from $D^{-}$) onto $\Delta_{2}=\{ (u,v)\in\mathbb{R}^{2}\dvtx u^{2}<2v
\}$. Moreover
$\phi$ is~$\mathrm{C}^{1}$ on $D^{+}\cup D^{-}$, and its Jacobian in
$(z,t)$ is
$2|z-t|=2\sqrt{2v-u^{2}}\neq0$. The change of variables given by
$\phi
$ yields
\[
I_+=\int_{\Delta_{2}} h(u,v)\frac{1}{2\sqrt{2v-u^{2}}} f \biggl(
\frac
{u+\sqrt{2v-u^{2}}}{2} \biggr) f \biggl(\frac{u-\sqrt
{2v-u^{2}}}{2} \biggr) \,du \,dv,
\]
and $I_{-}=I_{+}$. By adding theses two terms, we get the lemma.
\end{pf}

By\vspace*{1.5pt} Theorem~\ref{exp(-nI)} in the \hyperref[appA]{Appendix}, the expansion of $g_{n}$
holds as soon as there exists $q \in [1,+\infty[$ such that
$\smash{\widehat{f}_{2} \in\mathrm{L}^{q}(\mathbb{R}^{d})}$. However the
computation of
$\smash{\widehat{f}_{2}}$ is not feasible in general.
Proposition~\ref{exp(-nI)2} says that the previous condition is satisfied if there
exists $p \in \, ]1,2]$ such that $f_{2} \in\mathrm{L}^{p}(\mathbb
{R}^{d})$ so that
the expansion is true. Let us take a look at this:
\begin{eqnarray*}
&&\int_{\mathbb{R}^{2}}\bigl|f_{2}(u,v)\bigr|^{p} \,du \,dv
\\
&& \qquad =\int_{\mathbb{R}^{2}}\frac{f^{p} ((u+\sqrt{2v-u^{2}})/2
) f^{p}
((u-\sqrt{2v-u^{2}})/2 )}{(2v-u^{2})^{p/2}} \mathbh {1}_{u^{2}<2v}
\,du\, dv.
\end{eqnarray*}
Let us make the change of variables given by
\[
(u,v) \longmapsto(x,y)=\tfrac{1}{2}\bigl(u+\sqrt{2v-u^{2}},u+
\sqrt {2v-u^{2}}\bigr),
\]
which is a $\mathrm{C}^{1}$-diffeomorphism from $\Delta_{2}$ to
$D^{+}$ (see the
proof of the previous lemma) with Jacobian in $(u,v)$, $2\sqrt
{2v-u^{2}}=2(y-x)>0$:
\[
\int_{\mathbb{R}^{2}}\bigl|f_{2}(u,v)\bigr|^{p}\,du\,dv=
\int_{\mathbb
{R}^{2}}\frac
{f^{p}(x)f^{p}(y)}{(y-x)^{p}} 2(y-x)\mathbh{1}_{y>x}
\,dx\,dy.
\]
By symmetry in $x$ and $y$, we get
\[
\int_{\mathbb{R}^{2}}\bigl|f_{2}(u,v)\bigr|^{p}\,du\,dv=
\int_{\mathbb{R}
^{2}}f^{p}(x)f^{p}(y)|y-x|^{1-p}
\,dx\,dy.
\]
Then we get the following proposition:

%pr17 #&#
\begin{prop}\label{exp(-nI)dens}
Suppose that $\rho$ has a density $f$ with respect to the
Lebesgue measure on $\mathbb{R}$ such that, for some $p \in  \,]1,2]$,
\[
\int_{\mathbb{R}^2}f^{p}(x+y)f^{p}(y)|x|^{1-p}
\,dx\,dy<+\infty.
\]
Then, for $n$ large enough, $\tilde{\nu}_{n,\rho}$ has a density
$g_{n}$ with respect to the Lebesgue measure on $\mathbb{R}^{2}$ such
that, for
any compact subset $K_{I}$ of $A_{I}$, when $n$ goes to~$+\infty$,
uniformly over $(x,y)\in K_{I}$.
\[
g_{n}(x,y)\sim\frac{n}{2\pi} \bigl(\operatorname{det} \mathrm{D}_{(x,y)}^{2}I \bigr)^{1/2} e^{-nI(x,y)}.
\]
\end{prop}

Let us prove now Theorem~\ref{theoFluctuations}. Suppose that $\rho$
is a
probability measure on~$\mathbb{R}$ with an even density $f$ such that there
exist $v_{0}>0$ and $p \in \, ]1,2]$ such that
\[
\int_{\mathbb{R}}e^{v_{0}z^{2}}f(z)\,dz<+\infty\quad \mbox{and}\quad \int
_{\mathbb{R}^{2}}f^{p}(x+y)f^{p}(y)|x|^{1-p}
\,dx\,dy<+\infty.
\]
The first inequality implies that
\mbox{$\mathbb{R}\times\,]{-}\infty,v_{0}[\,\subset
D_{\Lambda}$} and thus $(0,0) \in\DD_{\Lambda}$. Moreover $\rho$ is symmetric
(since $f$ is even), and its support contains at least three points
(since $\rho$ has a density). Proposition~\ref{DLdeI-F} implies that
there exists $\delta>0$ such that
\renewcommand{\theequation}{*}
%e1 #&#
\begin{equation}
\label{DevI-F}
\forall(x,y) \in\mathrm{B}_{\delta} \qquad G(x,y)=I(x,y)-
\frac
{x^{2}}{2y} \geq \frac{(y-\sigma^{2})^{2}}{4(\mu_{4}-\sigma^{4})}+\frac{\mu
_{4}x^{4}}{24\sigma^{8}},
\end{equation}
where $\mu_{4}$ denotes the fourth moment of $\rho$ and $\mathrm
{B}_{\delta}$ the
open ball of radius $\delta$ centered at $(0,\sigma^{2})$. We can\vspace*{-1.5pt}
reduce $\delta$,
in order to have $\mathrm{B}_{\delta}\subset K_{I}$ where $K_{I}$ is
a compact
subset of $A_{I}$. Moreover
$A_{I} \subset\DD_{I} \subset
\Delta^{ *}$ thus $\mathrm{B}_{\delta}\cap\Delta^{
*}=\mathrm{B}_{\delta}$.

Let $n \in\mathbb{N}$, and let $f \dvtx\mathbb{R}\longrightarrow\mathbb
{R}$ be a bounded
continuous function. We have
\[
\mathbb{E}_{\tilde{\mu}_{n,\rho}} \biggl(f \biggl(\frac
{S_{n}}{n^{3/4}} \biggr) \biggr)=
\frac{1}{Z_{n}}\int_{\Delta^{ *}}f\bigl(xn^{1/4}\bigr)\exp
\biggl(\frac
{nx^{2}}{2y} \biggr)\,d\tilde{\nu}_{n,\rho}(x,y)=
\frac
{A_{n}+B_{n}}{Z_{n}},
\]
with
\begin{eqnarray*}
A_{n} &=& \int_{\mathrm{B}_{\delta}}f\bigl(xn^{1/4}\bigr)
\exp \biggl(\frac
{nx^{2}}{2y} \biggr) \,d\tilde{\nu}_{n,\rho}(x,y),
\\
B_{n} &=& \int_{\Delta^{ *}\cap\mathrm{B}^{c}_{\delta
}}f\bigl(xn^{1/4}\bigr)
\exp \biggl(\frac
{nx^{2}}{2y} \biggr)\,d\tilde{\nu}_{n,\rho}(x,y).
\end{eqnarray*}
In what follows, we introduce $e^{-nI(x,y)}$ in the expression of
$A_{n}$, in order to use Proposition~\ref{exp(-nI)dens}:
\[
A_{n}=n\int_{\mathrm{B}_{\delta}}f\bigl(xn^{1/4}
\bigr)e^{-nG(x,y)}H_{n}(x,y)\, dx\,dy,
\]
where we set $H_{n}(x,y)=e^{nI(x,y)}g_{n}(x,y)/n$. We define
\[
\mathrm{B}_{\delta,n}=\bigl\{ (x,y) \in\mathbb{R}^{2} :
x^{2}/\sqrt {n}+y^{2}/n \leq\delta^{2} \bigr\}.
\]
Let us make the change of variables given by $(x,y) \longmapsto
(xn^{-1/4},yn^{-1/2}+\sigma^{2})$, with Jacobian $n^{-3/4}$,
\begin{eqnarray*}
A_{n}&=& n^{1/4}\int_{\mathrm{B}_{\delta,n}}f(x)\exp
\biggl(-nG \biggl(\frac
{x}{n^{1/4}},\frac{y}{\sqrt{n}}+\sigma^{2}
\biggr) \biggr)
\\
&&\hspace*{42pt}{}\times H_{n} \biggl(\frac{x}{n^{1/4}},\frac{y}{\sqrt{n}}+\sigma
^{2} \biggr)\, dx\,dy.
\end{eqnarray*}
We check now that we can apply the dominated convergence theorem to
this integral. The uniform expansion of $g_{n}$ (see Proposition~\ref
{exp(-nI)dens}) means that for any $\alpha>0$, there exists $n_{0} \in
\mathbb{N}$
such that
\[
\forall(x,y) \in K_{I}, \forall n\geq n_{0}\qquad \bigl\llvert
H_{n}(x,y) 2\pi \bigl(\operatorname{det} \mathrm{D}_{(x,y)}^{2}I
\bigr)^{-1/2} -1\bigr\rrvert \leq\alpha.
\]
If $(x,y) \in\mathrm{B}_{\delta,n}$, then
$(x_{n},y_{n})=(xn^{-1/4},yn^{-1/2}+\sigma
^{2}) \in\mathrm{B}_{\delta} \subset K_{I}$. Thus for all $n\geq
n_{0}$ and
$(x,y) \in\mathrm{B}_{\delta,n}$,
\[
\biggl\llvert H_{n} \biggl(\frac{x}{n^{1/4}},\frac{y}{\sqrt{n}}+
\sigma ^{2} \biggr)2\pi \bigl(\operatorname{det} \mathrm{D}_{(x_{n},y_{n})}^{2}I
\bigr)^{-1/2} -1\biggr\rrvert \leq\alpha.
\]
Moreover $(x_{n},y_{n})$ goes to $(0,\sigma^{2})$. Thus by continuity,
\[
\bigl(\mathrm{D}_{(x_{n},y_{n})}^{2}I \bigr)^{-1/2} \mathop{\longrightarrow}_{n\to+\infty} \bigl(\mathrm {D}_{(0,\sigma^{2})}^{2}I
\bigr)^{-1/2} = \bigl(\mathrm{D}_{(0,0)}^{2}\Lambda
\bigr)^{1/2},
\]
whose determinant is equal to $\sqrt{\sigma^{2}(\mu_{4}-\sigma^{4})}$.
Therefore
\[
\mathbh{1}_{\mathrm{B}_{\delta,n}}(x,y)H_{n} \biggl(\frac
{x}{n^{1/4}},
\frac{y}{\sqrt{n}}+\sigma ^{2} \biggr)\mathop{\longrightarrow}_{n\to+\infty}
\bigl(4 \pi^{2} \sigma ^{2}\bigl(\mu_{4}-
\sigma^{4}\bigr) \bigr)^{-1/2}.
\]
The expansion of $G$ in the neighborhood of $(0,\sigma^{2})$ implies that
\[
\exp \biggl(-nG \biggl(\frac{x}{n^{1/4}},\frac{y}{\sqrt{n}}+\sigma
^{2} \biggr) \biggr) \mathop{\longrightarrow}_{n\to+\infty} \exp
\biggl(-\frac
{y^{2}}{2(\mu_{4}-\sigma^{4})}-\frac{\mu_{4}x^{4}}{12\sigma
^{8}} \biggr).
\]
Let us check that the integrand is dominated by an integrable function,
which is independent of $n$. The function
\[
(x,y) \longmapsto \bigl(\mathrm{D}_{(x,y)}^{2}I
\bigr)^{-1/2}
\]
is bounded on $\mathrm{B}_{\delta}$ by some $M_{\delta}>0$. The
uniform expansion of
$g_{n}$ implies that for all $(x,y) \in\mathrm{B}_{\delta}$,
$H_{n}(x,y)\leq C_{\delta
}$ for some constant $C_{\delta}>0$. Finally, the inequality~(\ref{DevI-F})
above yields
\begin{eqnarray*}
&& \mathbh{1}_{\mathrm{B}_{\delta,n}}(x,y)f(x)\exp \biggl(-nG \biggl(\frac{x}{n^{1/4}},
\frac
{y}{\sqrt{n}}+\sigma^{2} \biggr) \biggr)H_{n} \biggl(
\frac
{x}{n^{1/4}},\frac
{y}{\sqrt{n}}+\sigma^{2} \biggr)
\\
&& \qquad \leq\|f\|_{\infty}C_{\delta}\exp \biggl(-\frac{y^{2}}{4(\mu
_{4}-\sigma
^{4})}-
\frac{\mu_{4}x^{4}}{24\sigma^{8}} \biggr).
\end{eqnarray*}
The right term is an integrable function on $\mathbb{R}^{2}$; thus it follows
from the dominated convergence theorem that
\[
A_{n}\mathop{\sim}_{+\infty}n^{1/4}\int
_{\mathbb{R}^{2}}\frac
{f(x)}{\sqrt{2
\pi\sigma^{2}}\sqrt{2 \pi(\mu_{4}-\sigma^{4})}}\exp \biggl(-\frac
{y^{2}}{2(\mu
_{4}-\sigma^{4})}-
\frac{\mu_{4}x^{4}}{12\sigma^{8}} \biggr)\,dx \,dy.
\]
By Fubini's theorem, we get
\[
A_{n}\mathop{\sim}_{+\infty}\frac{n^{1/4}}{\sqrt{2\pi\sigma
^{2}}}\int
_{\mathbb{R}
}f(x)\exp \biggl(-\frac{\mu_{4}x^{4}}{12\sigma^{8}} \biggr)\, dx.
\]
Let us focus now on $B_{n}$. The distribution $\rho$ is symmetric, it has
a density and $(0,0)$ belongs to the interior of $D_{\Lambda}$; thus
Proposition~\ref{TypeVaradhan} implies that there exist $\varepsilon
>0$ and
$n_{0} \geq1$ such that for any $n \geq n_{0}$,
\[
\int_{\Delta^{ *}\cap\mathrm{B}^{c}_{\delta}}\exp \biggl(\frac
{nx^{2}}{2y} \biggr) \,d
\tilde{\nu}_{n,\rho}(x,y)\leq e^{-n\varepsilon},
\]
and thus $B_{n} \leq\|f\|_{\infty}e^{-n\varepsilon}$, so that
$B_{n}=o(n^{1/4})$. Therefore
\[
A_{n}+B_{n}\mathop{\sim}_{+\infty}\frac{n^{1/4}}{\sqrt{2\pi\sigma
^{2}}}
\int_{\mathbb{R}}f(x)\exp \biggl(-\frac{\mu
_{4}x^{4}}{12\sigma^{8}} \biggr)\, dx.
\]
Applying this to $f=1$, we get
\[
Z_{n}\mathop{\sim}_{+\infty}\frac{2n^{1/4}}{\sqrt{2\pi\sigma
^{2}}}\int
_{0}^{+\infty}\exp \biggl(-\frac{\mu_{4}x^{4}}{12\sigma^{8}} \biggr)
\,dx=\frac
{n^{1/4}}{\sqrt{2\pi\sigma^{2}}}\frac{1}{2} \biggl(\frac{12\sigma
^{8}}{\mu
_{4}}
\biggr)^{1/4}\Gamma \biggl(\frac{1}{4} \biggr),
\]
where we made the change of variables $y=\mu_{4}x^{4}/(12\sigma
^{8})$. Finally
\[
\mathbb{E}_{\tilde{\mu}_{n,\rho}} \biggl(f \biggl(\frac
{S_{n}}{n^{3/4}} \biggr) \biggr)
\mathop{\sim}_{+\infty} \biggl(\frac{4\mu_{4}}{3\sigma^{8}} \biggr)^{1/4}
\Gamma \biggl(\frac{1}{4} \biggr)^{-1} \int_{\mathbb{R}}f(x)
\exp \biggl(-\frac{\mu
_{4}x^{4}}{12\sigma^{8}} \biggr) \,dx.
\]
The ultimate change of variables $s=\mu_{4}^{1/4}x/\sigma^{2}$ gives us
Theorem~\ref{theoFluctuations}.

\begin{appendix}

\setcounter{theo}{0}
\renewcommand{\thetheo}{A.\arabic{theo}}
\section{General results on the Cram\'er transform}\label{appA}
%\subsection{}

We present here some general results on the Cram\'er transform of
distributions on $\mathbb{R}^{d}$.

A probability measure $\mathbb{R}$ is said to be degenerate if it is a Dirac
mass. The following definition generalizes this notion for measures on
$\mathbb{R}^{d}$:

%de1 #&#
\begin{defi}
A probability measure $\nu$ on $\mathbb{R}^{d}$, $d
\geq2$, is
said to be degenerate if its support is included in a hyperplane of
$\mathbb{R}
^{d}$; that is, there exists a hyperplane $\mathcal{H}$ of $\mathbb
{R}^{d}$ such that
$\nu(\mathcal{H})=1$.
\end{defi}

A first consequence of the nondegeneracy of $\nu$ is that its
covariance matrix is a symmetric positive definite matrix; see
Section~III.5 of~\cite{Feller2} for a proof.

From this point forward, we consider $\nu$ a nondegenerate probability
measure on $\mathbb{R}^{d}$. The log-Laplace $L$ of $\nu$ is defined
in $\mathbb{R}
^{d}$ by
\[
\forall\lambda\in\mathbb{R}^{d}\qquad  L(\lambda)=\ln\int
_{\mathbb{R}^{d}} e^{\langle\lambda,z
\rangle} \,d\nu(z),
\]
where $\langle \cdot,\cdot \rangle$ denotes the inner product in
$\mathbb{R}^{d}$.
It is a convex function on $\mathbb{R}^{d}$ which takes its values in
$]{-}\infty
,+\infty]$. The Fenchel--Legendre transform of $L$ is called the Cram\'
er transform of $\nu$ and is defined on $\mathbb{R}^d$ by
\[
\forall x \in\mathbb{R}^{d} \qquad J(x)=\sup_{\lambda\in\mathbb
{R}^{d}} \bigl(
\langle\lambda ,x\rangle-L(\lambda) \bigr).
\]
It is a nonnegative, convex\vspace*{-1.5pt} and lower semi-continuous function. We
denote by $D_{L}$ and $D_{J}$ the convex sets where $L$ and $J$ are
finite. Notice\vspace*{-1.5pt}that if $\DD_{L}$
is nonempty, then $L$ is $\mathrm{C}^{\infty }$ on $\DD_{L}$. We refer to Section~2.2~of~\cite{DZ},
Section~VII.5~of~\cite{Ellis} and Sections~25 and~26 of~\cite{Rockafellar} for the main results on $L$ and $J$. Cram\'er's theorem
(Theorem~\ref{Cramer} in the \hyperref[appA]{Appendix}) links~$J$ and the large
deviations of $S_{n}/n$ where $S_{n}$ is the sum of $n$ independent
random variables with common distribution $\nu$.

We are interested in the points $\lambda$ realizing the supremum defining
$J(x)$, for $x\in D_{J}$. We denote by $\mathcal{C}$ the closed convex
hull of
the support of $\nu$.

%le2 #&#
\begin{lem}\label{lemEnvConv}
Let\vspace*{-1.5pt} $\nu$ be a nondegenerate probability measure on
$\mathbb{R}
^{d}$. The interior of $\mathcal{C}$ is not empty and $\CC\subset D_{J}
\subset\mathcal{C}$. Moreover for any $x\in\CC$, the supremum defining
$J(x)$ is realized for some value $\lambda(x) \in D_{L}$.
\end{lem}

\begin{pf}
The nondegeneracy of $\nu$ means that its support is not
included in a hyperplane of $\mathbb{R}^{d}$. Therefore the support of
$\nu$
contains $d$ linearly\vspace*{-1.5pt} independent vectors, and the interior of the
convex hull of these vectors is nonempty. Thus $\CC$ is nonempty.

Suppose that $\mathcal{C}\neq\mathbb{R}^{d}$ (otherwise it is
immediate that $D_{J}
\subset\mathcal{C}$). Let $x \notin\mathcal{C}$. By the
Hahn--Banach theorem, there
exists $\lambda\in\mathbb{R}^{d}$ and $a \in\mathbb{R}$ such that
\[
\forall y \in\mathcal{C}\qquad \langle\lambda,y \rangle\leq a < \langle\lambda,x
\rangle.
\]
Since $\nu(\mathcal{C})=1$, Jensen's inequality implies that
\[
\forall t>0\qquad  J(x)\geq- \ln\int_{\mathcal{C}} \exp\bigl(t\langle
\lambda,y \rangle -t\langle\lambda,x \rangle\bigr) \,d\nu(y)\geq t\bigl(\langle
\lambda,x \rangle-a\bigr).
\]
Sending $t$ to $+\infty$, we conclude that $J(x)=+\infty$. Thus $D_{J}
\subset\mathcal{C}$.

Let $x \in\CC$, and
let $(\lambda_{n})_{n \in\mathbb{N}}$ be a sequence in $\mathbb{R}
^{d}$ such that
\begin{eqnarray*}
J(x)&=& \lim_{n \to+\infty} \biggl(\langle\lambda_{n},x
\rangle- \ln \int_{\mathbb{R}
^{d}}\exp\bigl(\langle\lambda_{n},z
\rangle\bigr) \,d\nu(z) \biggr)
\\
&=& -\ln \lim_{n \to
+\infty} \int_{\mathbb{R}^{d}}\exp\bigl(
\langle\lambda_{n},z-x\rangle \bigr) \,d\nu(z).
\end{eqnarray*}
We suppose that $|\lambda_{n}|$ goes to $+\infty$, and we show that
it leads
to a contradiction. For all $n \in\mathbb{N}$, we set $u_{n}=\lambda
_{n}|\lambda
_{n}|^{-1}$. Then $(u_{n})_{n \in\mathbb{N}}$ is a bounded sequence.
Thus, up
to the extraction of a subsequence, we might assume that it converges
to some vector $u \in\mathbb{R}^{d}$ whose norm is $1$. Let $v$
belong to the
support of $\nu$, and let $U$ be an open subset of $\mathbb{R}^{d}$ containing
$v$. We have then $\nu(U)>0$. Suppose that for any $z \in U$, $\langle
u,z-x \rangle>0$. Then, by Fatou's lemma,
\begin{eqnarray*}
+\infty&=& \int_{U}\liminf_{n \to+\infty}\exp
\bigl(|\lambda_{n}|\langle u_{n},z-x\rangle\bigr) \,d\nu (z)
\\
&\leq & \liminf_{n \to+\infty}\int_{U} \exp\bigl(|
\lambda_{n}|\langle u_{n},z-x\rangle\bigr) \,d\nu (z).
\end{eqnarray*}
Hence
\[
\exp\bigl(-J(x)\bigr)=\lim_{n \to+\infty}\int_{\mathbb{R}^{d}}
\exp \bigl(|\lambda_{n}|\langle u_{n},z-x\rangle\bigr) \,d
\nu(z)=+\infty.
\]
Thus $J(x)=- \infty$, which is absurd since $J$ is a nonnegative
function. We conclude that for all $v$ in the support of $\nu$ and for
any open subset $U$ of $\mathbb{R}^{d}$ containing $v$, there exists
$z \in U$
such that
$\langle u,z-x \rangle\leq0$. It follows that, for any $v$ in the
support of $\nu$, $\langle u,v \rangle\leq\langle u,x \rangle$. This
inequality is stable by convex combinations, thus
\[
\forall y \in\mathcal{C}\qquad \langle u,y \rangle\leq\langle u,x \rangle.
\]
Since $x \in\CC$,
there exists a ball $\mathrm{B}_{x}$ centered at $x$ and
contained in $\mathcal{C}$. Thus there exists $y_{0} \in\mathrm
{B}_{x}$ such that
$\langle u,y_{0} \rangle> \langle u,x \rangle$, which is absurd.
Therefore $(\lambda_{n})_{n \in\mathbb{N}}$ is a bounded sequence.
Hence there
exists a subsequence $(\lambda_{\phi(n)})_{n \in\mathbb{N}}$ and
$\lambda(x) \in\mathbb{R}^{d}$
such that $\lambda_{\phi(n)}$ goes to $\lambda(x)$. By Fatou's lemma,
\begin{eqnarray*}
J(x)&=&\bigl\langle\lambda(x),x \bigr\rangle-\ln\lim_{n \to+\infty}\int
_{\mathbb{R}^{d}}\exp \bigl(\langle\lambda_{n},z\rangle\bigr)
\,d\nu(z)
\\
& \leq & \bigl\langle\lambda(x),x \bigr\rangle-\ln\int_{\mathbb
{R}^{d}}
\liminf_{n \to+\infty}\exp\bigl(\langle\lambda _{n},z\rangle
\bigr) \,d\nu(z)
\\
& =& \bigl\langle\lambda(x),x \bigr\rangle-\ln\int_{\mathbb{R}^{d}}\exp
\bigl(\bigl\langle\lambda(x),z\bigr\rangle\bigr) \,d\nu(z)\leq J(x).
\end{eqnarray*}
Thus $J(x)=\langle\lambda(x),x \rangle-L(\lambda(x))$.\vspace*{-1.5pt} Since
$L(\lambda(x))\neq-\infty
$, this formula implies that $J(x)<+\infty$ and thus that $\CC\subset D_{J}$. Moreover if $L(\lambda(x))=+\infty$, then
$J(x)=-\infty$,
which is absurd. Therefore $L(\lambda(x))<\infty$. This shows that the
supremum defining $J(x)$ is realized at a point $\lambda(x)$ with\vspace*{-1.5pt}
$\Lambda(\lambda
(x))<+\infty$.
\end{pf}

If $D_{L}$ is an open\vspace*{-1.5pt} subset of $\mathbb{R}^{d}$, then for all $(x,y)
\in\DD
_{J}=\CC$, the supremum
defining $J(x)$ is realized at some $\lambda(x)
\in\DD_{L}$. This is the case
when the support of $\nu$ is bounded,
and also for the distribution\vspace*{1pt} $\nu_{\rho}$ when $\rho$ is the
Gaussian $\mathcal{N}
(0,\sigma^{2})$, where we have then $D_{L}=\mathbb{R} \times\,
]{-}\infty,1/(2\sigma^{2})[$.

Now we study the smoothness of $J$.

%no1 #&#
\begin{nota*}
If $f$ is a differentiable function on an open subset
$U$ of $\mathbb{R}^{d}$, we denote by $\mathrm{D}_{x}f$ the
differential of $f$
at $x \in U$. If $f$ is real-valued, we denote:
\begin{longlist}[$\star$]
\item[$\star$]  $\mathrm{D}^{2}_{x}f$ its second differential at $x\in U$
(considered as a matrix of size $d\times d$).

\item[$\star$]  $\nabla f$ the function $U \longrightarrow\mathbb{R}^{d}$
such that
\[
\forall x \in U, \forall y\in\mathbb{R}^{d} \qquad \bigl\langle\nabla f(x),y
\bigr\rangle=\mathrm{D}_{x}f(y).
\]
\end{longlist}
\end{nota*}

We define the admissible domain of $J$:

%de3 #&#
\begin{defi}\label{adm}
Let $\nu$ be a nondegenerate probability measure on
$\mathbb{R}
^{d}$ such that the interior of $D_{L}$ is nonempty. The\vspace*{-1pt} admissible
domain of $J$ is the set $A_{J}=\nabla L (\DD_{L})$.
\end{defi}

The following proposition states that $A_{J}$, the admissible domain of
$J$, is an open subset of $\mathbb{R}^{d}$, and that $J$ is $\mathrm
{C}^{\infty}$ on $A_{J}$.

%pr4 #&#
\begin{prop}\label{Dadmissible}
Let $\nu$ be a nondegenerate probability measure on
$\mathbb{R}^{d}$ such that the interior of $D_{L}$ is nonempty. Let $A_{J}$ be
the admissible domain of $J$. We have:
\begin{longlist}[(a)]
\item[(a)]  The function $\nabla L$ is a $\mathrm{C}^{\infty }$-diffeomorphism from $\DD_{L}$ to $A_{J}$. Moreover
\[
A_{J} \subset D_{J}=\bigl\{ x \in\mathbb{R}^{d}\dvtx J(x)<+
\infty \bigr\}.
\]

\item[(b)]  Denote by $\lambda$ the inverse $\mathrm{C}^{\infty }$-diffeomorphism of $\nabla L$. Then the function $J$
is~$\mathrm{C}^{\infty}$
on $A_{J}$ and for any $x \in A_{J}$,
\begin{eqnarray*}
J(x)&=& \bigl\langle x,\lambda(x)\bigr\rangle-L\bigl(\lambda(x)\bigr),
\\
\nabla J(x) &=& (\nabla L)^{-1}(x) = \lambda(x)\quad \mbox{and}\quad \mathrm
{D}^{2}_{x}J= \bigl(\mathrm{D}^{2}_{\lambda(x)}L
\bigr)^{-1}.
\end{eqnarray*}

\item[(c)] If $D_{L}$ is an open subset of $\mathbb
{R}^{d}$, then
$A_{J}=\DD_{J}=\CC$ where $\mathcal{C}$ denotes the
convex hull of the
support of $\nu$.
\end{longlist}
\end{prop}

\begin{pf}
The points (a) and (b)  are proved in
Section~2 of~\cite{Baldi}, Section~1.5 of~\cite{Borovkov} and
Section~26 of~\cite{Rockafellar}; see also Section~VIII.4 of~\cite{Ellis} in the case where $D_{L}=\mathbb{R}^d$. Let us prove point
(c). If
$D_{L}$ is an open subset of $\mathbb{R}^{d}$, then Lemma~\ref{lemEnvConv}
implies that for $x\in\CC=\DD_{J}$, the supremum\vspace*{-1.5pt}
defining $J(x)$ is
realized at some point $\lambda(x) \in D_{L}=\DD_{L}$. The function $L$ is
differentiable at $\lambda(x)$, and point (b) yields that
\[
x=\nabla L\bigl(\lambda(x)\bigr) \in\Lambda(\DD_{L})=A_{J}.
\]
Thus $\DD_{J} \subset A_{J}$.
Finally $A_{J}\subset D_{J}$, and
$A_{J}$ is open; thus $A_{J}=\DD_{J}=\CC$. This
proves (c).
\end{pf}

Let $\nu$ be a probability distribution on $\mathbb{R}^{d}$ having a density
with respect to the Lebesgue measure, and let $S_{n}$ be the sum of $n$
independent and identically distributed random variables with
distribution $\nu$. The following theorem states that, under some
hypothesis allowing the Fourier inversion, the density of the
distribution of $S_{n}/n$ is asymptotically a function of $J$, the
Cram\'er transform of $\nu$. We refer to Section~3 of the article of
Andriani and Baldi~\cite{Baldi} for a proof.

%th5 #&#
\begin{theo}\label{exp(-nI)}
Let\vspace*{-1.5pt} $\nu$ be a nondegenerate probability measure on
$\mathbb{R}^{d}$.
We denote by $L$ its log-Laplace and by $J$ its Cram\'er
transform. Suppose that $\DD_{L}\neq\varnothing$ and that there
exists $n_{0} \geq1$ such that
\[
\widehat{\nu^{*n_{0}}} \in\mathrm{L}^{1}\bigl(
\mathbb{R}^{d}\bigr).
\]
We denote by $A_{J}$ the admissible domain of $J$. Let $(X_{n})_{n \geq
1}$ be a sequence of independent and identically distributed random
variables with distribution~$\nu$. For any $n\geq n_{0}$, the random
variable $\overline{X}_{n}=(X_{1}+\cdots+X_{n})/n$ has a density~$g_{n}$
with respect to the Lebesgue measure on $\mathbb{R}^{d}$. If $K_{J}$
is a
compact subset of $A_{J}$, then uniformly over $x\in K_{J}$, when $n$
goes to $+\infty$,
\[
g_{n}(x)\sim \biggl(\frac{n}{2\pi} \biggr)^{d/2} \bigl(
\operatorname{det} \mathrm{D}_{x}^{2}J \bigr)^{1/2}
e^{-nJ(x)}.
\]
\end{theo}

%pr6 #&#
\begin{prop}\label{exp(-nI)2}
Let $\nu$ be\vspace*{-1.5pt} a nondegenerate probability measure on
$\mathbb{R}
^{d}$ such that $\DD_{L}\neq
\varnothing$. If there exists $m \in\mathbb{N}$
and $p \in  \,]1,2]$ such that $\nu^{*m}$ has a density $f_{m} \in
\mathrm{L}
^{p}(\mathbb{R}^{d})$, then the hypotheses of Theorem~\ref{exp(-nI)}
are verified.
\end{prop}

\begin{pf}
The Hausdorff--Young inequality (see
Theorem~1.2.1 of~\cite{BL}) implies that $\smash{\widehat{f}_{m} \in
\mathrm{L}
^{r}(\mathbb{R}^{d})}$, with $r=p/(p-1)$. Moreover $\smash{\widehat
{f}_{m}}$ is
bounded, so $\smash{\widehat{f}_{m} \in\mathrm{L}^{q}(\mathbb
{R}^{d})}$, where $q$ is
a positive integer larger than $r$. Therefore
\[
\widehat{\nu^{*mq}}= \bigl(\widehat{\nu^{*m}}
\bigr)^{q}= (\widehat {f}_{m} )^{q}\in
\mathrm{L}^{1}\bigl(\mathbb{R}^{d}\bigr).
\]
Hence the hypotheses of the theorem are verified with $n_{0}=mq$.
\end{pf}

\section{Some results on large deviations}\label{appB}
%\subsection{}
\renewcommand{\thetheo}{B.\arabic{theo}}
\setcounter{theo}{0}

Let $(\mathcal{X},\mathcal{B})$ be a topological space. We refer to the
Section~1.2 of~\cite{DZ} for the two following definitions:

%de1 #&#
\begin{defi}
A rate function on $\mathcal{X}$ is a nonnegative map $J$
defined on~$\mathcal{X}$ and which is lower semi-continuous; that is,
for any
$\alpha>0$, the level set $\{ x \in\mathcal{X}\dvtx  J(x) \leq\alpha
\}$ is a
closed subset of $\mathcal{X}$. A good rate function is a rate
function for
which all these level sets are compact sets of $\mathcal{X}$.
\end{defi}

%de2 #&#
\begin{defi}
A sequence $(\mu_n)_{n\geq1}$ of probability measures
on $\mathcal{X}$ satisfies a large deviation principle with speed $n$
and which
is governed by the rate function~$J$ if, for any $A \in\mathcal{B}$,
\begin{eqnarray*}
-\inf \bigl\{ J(x)\dvtx x \in\AAA \bigr\} &\leq & \liminf
_{n \to+\infty}\frac{1}{n}\ln\mu_n(A)
\\
& \leq & \limsup_{n \to+\infty}\frac{1}{n}\ln\mu_n(A)
\leq-\inf \bigl\{ J(x)\dvtx x \in \overline{A} \bigr\}.
\end{eqnarray*}
\end{defi}

The following lemma is a variant of the upper bound of Varadhan's
lemma; see Lemma~4.3.6 of~\cite{DZ}.

%le3 #&#
\begin{lem}\label{VaradhanUpper}
Let $\mathcal{X}$ be a regular topological space endowed
with its
Borel $\sigma$-field $\mathcal{B}$. Let $(\nu_{n})_{n \geq1}$ be a
sequence of
probability measures defined on $(\mathcal{X},\mathcal{B})$ which
satisfies a large
deviation principle with speed $n$, governed by the good rate function
$J$. For any bounded continuous function $f \dvtx \mathcal{X}\longrightarrow\mathbb{R}$,
we have for any closed subset $A$ of $\mathcal{X}$,
\[
\limsup_{n \to+\infty}\frac{1}{n} \ln\int_{A}
e^{nf(x)} \,d\nu_{n}(x)\leq\sup_{x
\in A}
\bigl(f(x)-J(x)\bigr).
\]
\end{lem}

We end this \hyperref[appA]{Appendix} with the Cram\'er theorem in $\mathbb{R}^d$ (see
Theorem~2.2.30 of~\cite{DZ}):

%th4 #&#
\begin{theo}[(Cram\'er)]\label{Cramer}
Let $\nu$ be a probability measure on $\mathbb{R}^d$,
$d\geq1$. We denote by $L$ its log-Laplace and by $J$ its Cram\'er
transform. Let $(X_n)_{n\geq1}$ be a sequence of independent random
variables with common law $\nu$. We define
\[
\forall n\geq1 \qquad S_n=X_1+\cdots+X_n.
\]
If $L$ is finite in the neighborhood of~$0$, then $J$ is a good rate
function, and the sequence of the laws of $S_n/n$, $n\geq1$ satisfies
the large deviation principle with speed $n$ and governed by $J$.
\end{theo}
\end{appendix}

\section*{Acknowledgments}
We thank two anonymous referees for their comments which helped to
improve the presentation of the paper.

% imsref loaded by daiva.urboniene, 2014-12-11 11:06:01
%
% imsref loaded by daiva.urboniene, 2014-12-17 10:15:33

%
%\section{}
%\end{appendix}

%\section*{Acknowledgments}

%\begin{supplement}[id=suppA]
%\slink[doi]{10.1214/00-AOPXXXXSUPP} %[doi,text={...}] - jei reikia
%suskaldyti doi
%\sfilename{aopXXXX\_supp.pdf}
%\end{supplement}

%\begin{thebibliography}{99}
%\bibitem[\protect\citeauthoryear{}{}]{r1}
%\bibitem{r1}
%\end{thebibliography}

\printaddresses
\end{document}